\documentclass[11pt]{amsart}
\usepackage{amsthm}
\usepackage{amssymb}
\usepackage[usenames]{color}
\usepackage{latexsym}
\usepackage{graphicx}
\usepackage{enumerate} 
\usepackage{graphics} 
\usepackage{epstopdf} 
\usepackage[all]{xy}
\usepackage{verbatim} 

\usepackage{setspace}

\usepackage[all]{xy}
\input epsf

\numberwithin{figure}{section}
\newtheorem{thm}{Theorem}[section]
\newtheorem{lem}[thm]{Lemma}
\newtheorem{cor}[thm]{Corollary}
\newtheorem{prop}[thm]{Proposition}

\newtheorem{defn}[thm]{Definition}

\newtheorem*{Wolcott}{Wolcott's Theorem} 

\newtheorem*{Schubert}{Schubert's Theorem} 

\newtheorem*{Thistlethwaite}{Thistlethwaite's Theorem} 

\newtheorem*{Sawollek}{Sawollek's Theorem}

\title{Ravels Arising from Montesinos Tangles}

\author{Erica Flapan}

\address{Department of Mathematics, Pomona College, Claremont, CA 91711, USA}

\author{Allison N. Miller}
\address{Department of Mathematics, University of Texas, Austin, TX 78712, USA}

\thanks{The first author was supported in part by NSF Grant DMS-1607744}

\begin{document}
\begin{abstract} A ravel is a spatial graph which is non-planar but contains no non-trivial knots or links.  We characterize when a Montesinos tangle can become a ravel as the result of vertex closure with and without replacing some number of crossings by vertices.
\end{abstract}

\maketitle

\section{Introduction}
One of the earliest results in spatial graph theory was the discovery by Suzuki \cite{suzuki} in 1970 of an embedding of an abstractly planar graph which had the property that it was non-planar but every subgraph of the embedding was planar.  Note that a graph is said to be {\it abstractly planar} if it can be embedded in $\mathbb{R}^2$, and a particular  embedding of a graph in $\mathbb{R}^3$ is said to be {\it planar} if there is an ambient isotopy of it into $\mathbb{R}^2\subseteq \mathbb{R}^3$.  Two years after Suzuki's result, Kinoshita \cite{kinoshita} found an embedding of a $\theta_3$ graph which had this property. Many results about such embeddings have since been obtained, though several different terms are used to refer to them.  In particular, we have the following definition.

\begin{defn}  An embedding $G$ of an abstractly planar graph in $\mathbb{R}^3$  is said to be \emph{almost unknotted} (equivalently \emph{almost trivial}, \emph{minimally knotted}, or \emph{Brunnian}) if  $G$ is non-planar but $G- \{e\}$ is planar  for any edge $e$ of $G$. 
 \end{defn}
 
One of the most significant results in the study of almost unknotted graphs is the result obtained by Kawauchi \cite{kawauchi} and Wu \cite{wu} that every abstractly planar graph without valence one vertices has an almost unknotted embedding.

We are now interested in a larger class of embedded graphs defined below. 
 
\begin{defn}{}
An embedding $G$ of an abstractly planar graph in $\mathbb{R}^3$  is said to be a \emph{ravel} if $G$ is non-planar but  contains no non-trivial knots or links. 
\end{defn}

Any almost unknotted graph $G$ is a ravel, unless $G$ is topologically a non-trivial knot or a Brunnian link.  However, the converse is not true. For example, starting with an almost unknotted embedding of a graph $G$, add an additional edge $e'$ parallel to an existing edge $e$ to get a new embedded graph $G'$. Since $G$ was almost unknotted, $G'$ will contain no non-trivial knots or links.  However, the removal of the edge $e'$ will not make $G'$ planar, and hence $G'$ is a ravel that is not almost unknotted.

The term {\it ravel} was originally coined as a way to describe hypothetical molecular structures whose complexity results from
``an entanglement of edges around a vertex that contains no knots or links" \cite{castle08}.  The first molecular ravel to be identified was a metal-ligand complex synthesized by Feng Li et al in 2011 \cite{Li2011}.  In order to formalize the notion of entanglement about a vertex, we require the ``entanglement'' to be properly embedded in a ball.  If we bring the endpoints of the edges in the boundary sphere together into a single vertex, we obtain a spatial graph which is known as the {\it vertex closure} $V(T)$ of the entanglement $T$.  In Figure~\ref{chemravel}, we illustrate an entanglement whose vertex closure is a ravel. 

\begin{figure}[http]
\begin{center}
\includegraphics[width=5cm]{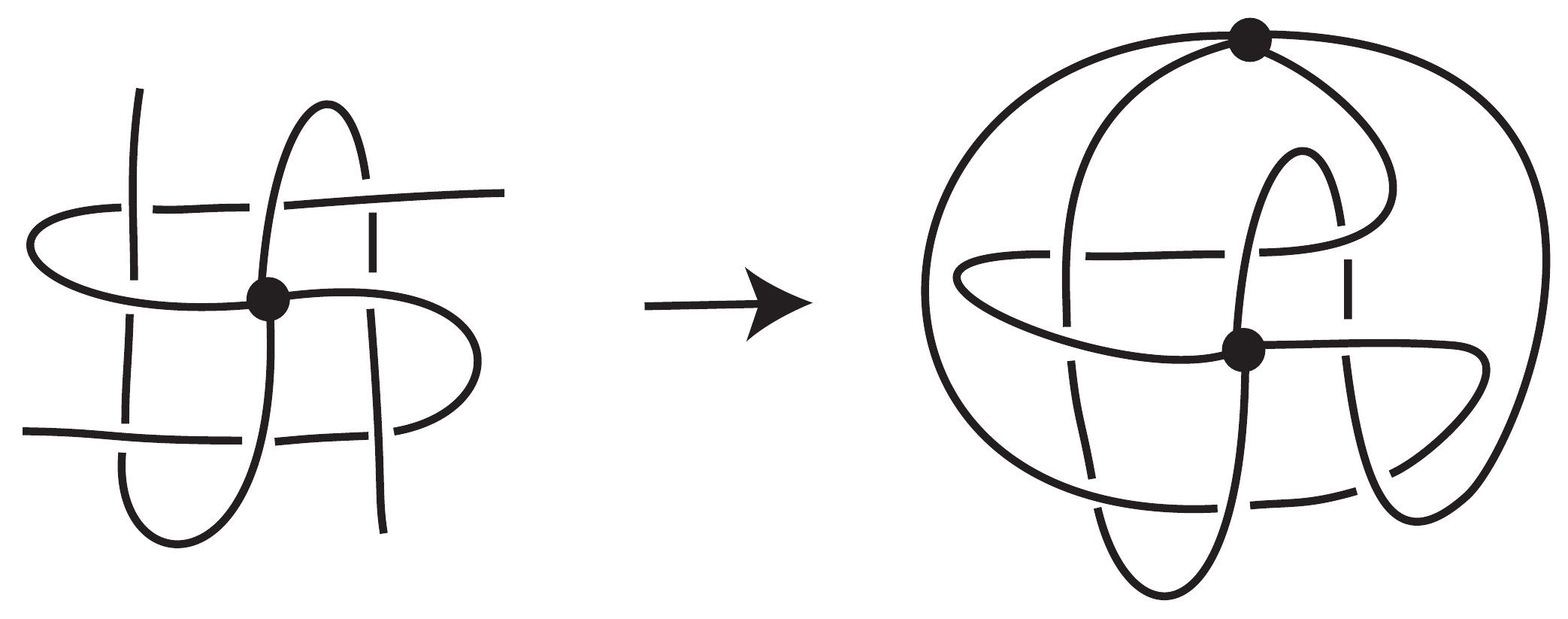}
\caption{The entanglement on the left becomes a ravel when the endpoints are brought together on the right.}
\label{chemravel}
\end{center}
\end{figure}

In this paper, we characterize when a Montesinos tangle can become a ravel as the result of vertex closure with and without replacing some number of crossings by vertices.  In particular, our main results are the following.

\newtheorem*{theorem1}{Theorem \ref{theorem1}}
\begin{theorem1}
Let $T=T_1+\dots +T_n$ be a Montesinos tangle such that $n$ is minimal and not both $T_1$ and $T_n$ are trivial vertical tangles.  If $T$ is rational, then the vertex closure $V(T)$ is planar. If $T$ is not rational and some rational subtangle $T_i$ has $\infty$-parity, then $V(T)$ contains a non-trivial knot or link. Otherwise, $V(T)$ is a ravel. \end{theorem1}

\newtheorem*{monthm}{Theorem \ref{monthm}}
\begin{monthm}
Let $T=T_1+\dots +T_n$ be a projection of a Montesinos tangle in standard form, and let $T'$ be obtained from $T$ by replacing at least one crossing by a vertex. Then the vertex closure $V(T')$ is a ravel if and only if $T'$ is an exceptional vertex insertion. 
\end{monthm}

While we postpone defining an exceptional vertex insertion until Section~\ref{exception}, Theorem~\ref{monthm} has the following more easily stated corollary.

\newtheorem*{arborescentclosure}{Corollary \ref{arborescentclosure}}
\begin{arborescentclosure}
Let $T=T_1+\dots +T_n$ be a projection of a Montesinos tangle in standard form, and let $T'$ be obtained from $T$ by replacing at least one crossing with a vertex. If the vertex closure $V(T')$ is a ravel, then precisely one $T_i$ has $\infty$-parity.
\end{arborescentclosure}
\bigskip


\section{Background}

For completeness we include some well-known definitions and results about knots, links, and tangles.

\begin{defn} \label{ratdef}
A $2$-string tangle $T$ in a ball $B$ is said to be \emph{rational} if there is an ambient isotopy of $B$ setwise fixing $\partial B$ that takes $T$ to a trivial tangle. 
\end{defn}

\begin{defn}
The \emph {sum} and \emph{product} of tangles $R$ and $S$ are shown in Figure~\ref{tanglesum}. 
\end{defn}

\begin{figure}[http]
\begin{center}
\includegraphics[width=5cm]{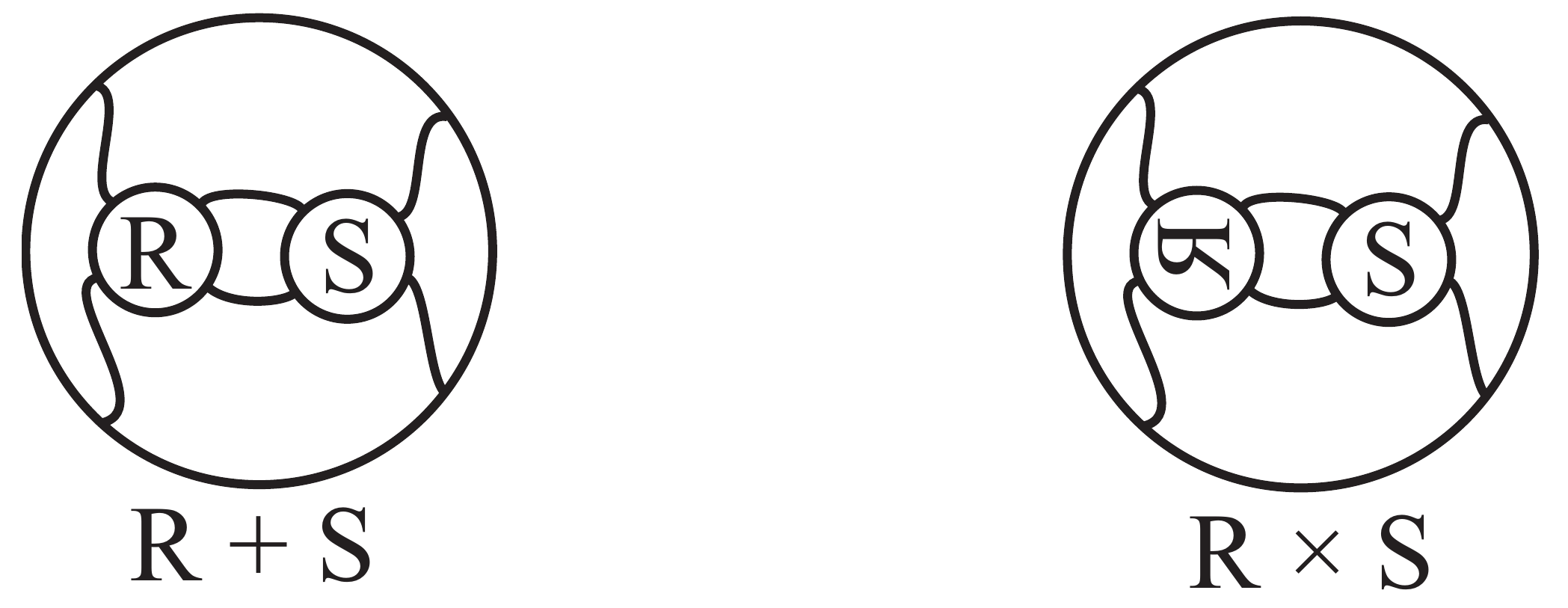} 
\caption{The sum and product of tangles $R$ and $S$.}
\label{tanglesum}
\end{center}
\end{figure}

\begin{defn}
A tangle is said to be \emph{Montesinos} if it can be written as the sum of finitely many rational tangles. 
A tangle is said to be \emph{algebraic} or \emph{arborescent} if it can be written in terms of sums and products of finitely many rational tangles. 
\end{defn}

Note that Montesinos tangles and algebraic tangles are not necessarily $2$-string tangles since they may contain simple closed curves in addition to the two strings.

\begin{defn}
Let $T$ be a 2-string tangle.  The knot or link obtained by joining the NE and SE points together and the NW and SW points together is called the {\em denominator closure of $T$}, and denoted by $D(T)$.  The knot or link obtained by joining the NW and NE points together and the SW and SE points together is called the {\em numerator closure of $T$}, and denoted by $N(T)$. 
 \end{defn}

\begin{defn}
A $2$-string tangle is said to have \emph{$\infty$-parity} if the NW and SW boundary points are on the same strand, and  \emph{$0$-parity} if the NW and NE boundary points are on the same strand. \end{defn}

\begin{defn}  Any tangle obtained from a trivial horizontal tangle by twisting together the NE and SE ends is said to be a \emph{horizontal} tangle.
\end{defn}

We will use the following results repeatedly in our proofs.

\begin{Wolcott}\label{denom} \cite{Wolcott86}
Let $T$ be a rational tangle.  Then $D(T)$ is the unknot if and only if $T$ is a horizontal tangle; and $D(T)$ is an unlink if and only if $T$ is a trivial vertical tangle. 
\end{Wolcott}

\begin{Schubert}\cite{schubert} \label{spherethm2} 
Let $L_1$ and $L_2$ be knots or links.  Then $L_1\#L_2$ is trivial if and only if both $L_1$ and $L_2$ are trivial.\end{Schubert}

\begin{Thistlethwaite} \label{thistle}  \cite{thistle87}
A reduced alternating projection of a link has the minimum number of crossings.
\end{Thistlethwaite}

\bigskip


\section{Vertex closure of rational and Montesinos tangles}

\begin{defn}
Let $T$ be a 1-string or 2-string tangle (possibly with additional closed components) in a ball $B$. The embedded graph $V(T)$ obtained by bringing the endpoints of the string(s) together into a single vertex $w$ in $\partial B$ is said to be the \emph{vertex closure of $T$} and $w$ is said to be the \emph{closing vertex}. 
\end{defn}

We begin with the following observation about when the vertex closure of a tangle is planar. 

\begin{lem}\label{rat0lem}
Let $T$ be a 2-string tangle with tangle ball $B$. Then the vertex closure $V(T)$ of $T$ is planar if and only if $T$ is rational.
\end{lem}

\begin{proof}
It follows from the definition of rational that $T$ is rational if and only if it can be made planar by moving the endpoints of the strands of $T$ around in $\partial B$. However, moving the endpoints of the strands around in $\partial B$ corresponds  to moving the edges of $V(T)$ about the closing vertex.  So $T$ is rational if and only if $V(T)$ can be made planar by moving the edges of $V(T)$ about  the vertex. 
\end{proof}

The following theorem characterizes when the vertex closure of a Montesinos tangle is a ravel.

\begin{thm}\label{theorem1}
Let $T=T_1+\dots +T_n$ be a Montesinos tangle such that $n$ is minimal and not both $T_1$ and $T_n$ are trivial vertical tangles.  If $T$ is rational, then the vertex closure $V(T)$ is planar.  If $T$ is not rational and some rational subtangle $T_i$ has $\infty$-parity, then $V(T)$ contains a non-trivial knot or link. Otherwise, $V(T)$ is a ravel. 
\end{thm}

\begin{proof}  We know by Lemma~\ref{rat0lem} that if $T$ is rational, then $V(T)$ is planar.  So we assume that $n>1$.  Since $n$ is minimal, none of the $T_i$ is horizontal. Without loss of generality we can assume that $T_n$ is not a trivial vertical tangle.

First suppose that at least one of the rational subtangles has $\infty$-parity.  Let $T_i$ be the rightmost such tangle in the  sum $T=T_1+\dots +T_n$.  Thus both ends of the NE-SE strand of $T_i$ can be extended to the right until they are joined together at the closing vertex $w$, giving us a loop $L_1$ (illustrated in grey on the left in Figure~\ref{closureprojection}), though we may have $i=n$.

\begin{figure}[http]
 \begin{center}
\includegraphics[width=12cm]{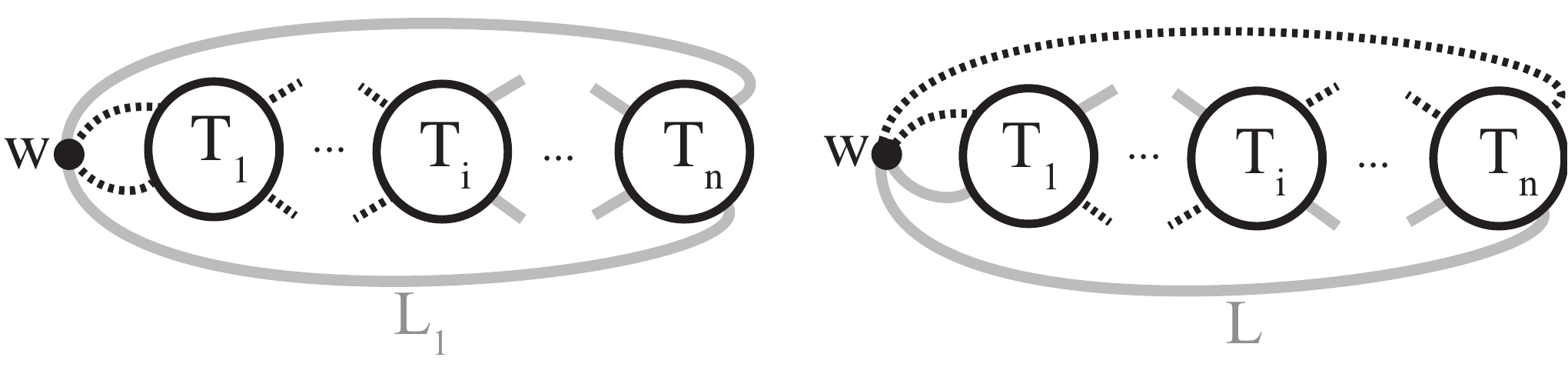}
\caption{On the left $T_i$ has $\infty$-parity and on the right no subtangle has $\infty$-parity. }
\label{closureprojection}
\end{center}
\end{figure}

  Suppose that $i<n$.  Then the loop $L_1$ is a connected sum of  the denominator closure $D(T_{n})$ together with a (possibly trivial) knot to the left.  Since $T_{n}$ does not have $\infty$-parity,  $D(T_{n})$ has a single component.  Now since $T_n$ is not horizontal, by Wolcott's Theorem $D(T_{n})$ is a non-trivial knot, and hence by  Schubert's Theorem, $L_1$ is a non-trivial knot contained in $V(T)$.   
  
Next suppose that $i=n$ and $T_n$ is the only $T_j$ with $\infty$-parity. Then we can extend both ends of the NW-SW strand of $T_n$ to the left until they join together at the closing vertex $w$.  We denote this loop by $L_2$.  Since $n>1$, $L_2$ is the connected sum of $D(T_1)$ and another (possibly trivial) knot.  Now since $T_1$ is not horizontal, by Walcott's Theorem and Schubert's Theorem $L_2$ is a non-trivial knot in $V(T)$.

Now suppose that some rational subtangle in addition to $T_n$ has $\infty$-parity.   Let $T_i$ be the tangle with $\infty$-parity that is closest to $T_n$. Then both ends of the NE-SE strand of $T_n$ can be extended rightward to $w$ to obtain a loop $L_1$; and both ends of the NW-SW strand of $T_n$ can be extended leftward until they are joined together in $T_i$, to obtain a loop $L_2$. Then the link $L=L_1\cup L_2$ is the connected sum of $D(T_n)$ with some possibly trivial knot.  Now since $T_n$ is  not a trivial vertical tangle, by Wolcott's Theorem and Schubert's Theorem, $L$ is a non-trivial link in $V(T)$.

Finally, suppose that no $T_i$ has $\infty$-parity.  Then $V(T)$ is an embedding of the wedge of two circles and hence $V(T)$ cannot contain a two component link.  Let $L$ denote the vertex closure of a single strand of $T$.  Since no $T_i$ has $\infty$-parity, $L$ passes through each $T_i$ exactly once, as illustrated by the grey arcs  on the right side of Figure~\ref{closureprojection}.  Now since each $T_i$ is rational, by Lemma \ref{rat0lem}, each individual $V(T_i)$ is planar.  It follows that the vertex closure of each of the single strands $T_i\cap L$ is unknotted. Now the loop $L$ is the connected sum of the loops $V(T_1 \cap L)$,  \dots, $V(T_n \cap L)$, each of which is unknotted.  Hence $L$ is a trivial knot.  Thus $V(T)$ contains no non-trivial knots or links.  However, since $T$ is not rational, we know by Lemma \ref{rat0lem} that  $V(T)$ is non-planar. Hence in this case $V(T)$ is a ravel.  
\end{proof}

The tangle in Figure~\ref{twoinfty} illustrates why Theorem~\ref{theorem1} has the hypothesis that not both $T_1$ and $T_n$ are trivial vertical tangles.  In this case, $V(T_1+T_2)$ is planar even though $T_1+T_2$ is a non-rational Montesinos tangle.

\begin{figure}[http]
\begin{center}
\includegraphics[width=4.5cm]{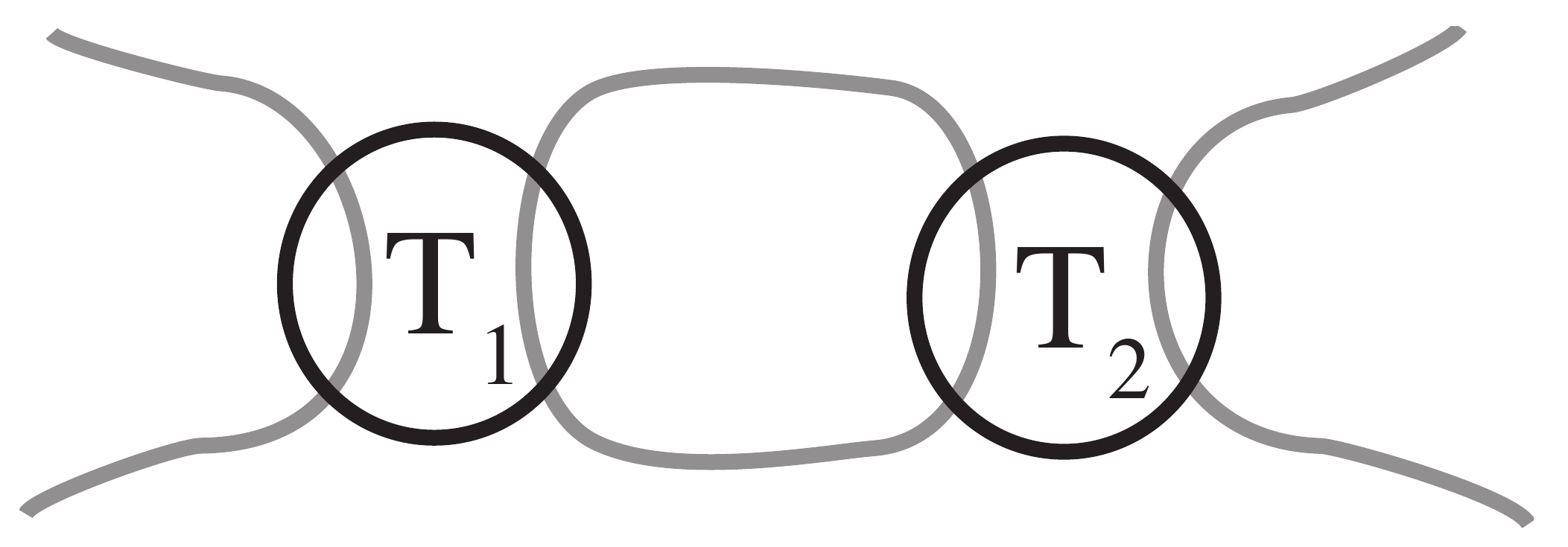}
\caption{A Montesinos tangle where both $T_1$ and $T_2$ are trivial vertical tangles.}
\label{twoinfty}
\end{center}
\end{figure}

\begin{cor}\label{arborescentclosure}
Let $T$ be a non-rational algebraic tangle written as the sum and product of rational tangles $T_1, \dots, T_n$ and either $n>2$ or not both $T_1$ and $T_2$ are trivial vertical tangles.  If each strand of $T$ passes through each $T_i$ exactly once, then $V(T)$ is a ravel.
\end{cor}

\begin{proof}  Observe that by our hypotheses, $V(T)$ must be a wedge of two circles.  Thus the argument is analogous to the last case in the proof of Theorem~\ref{theorem1}.
\end{proof}

The algebraic tangle $T$ in Figure~\ref{arbravel} illustrates that the converse of Corollary \ref{arborescentclosure} does not hold. 
\begin{figure}[http]
\begin{center}
\includegraphics[width=4cm]{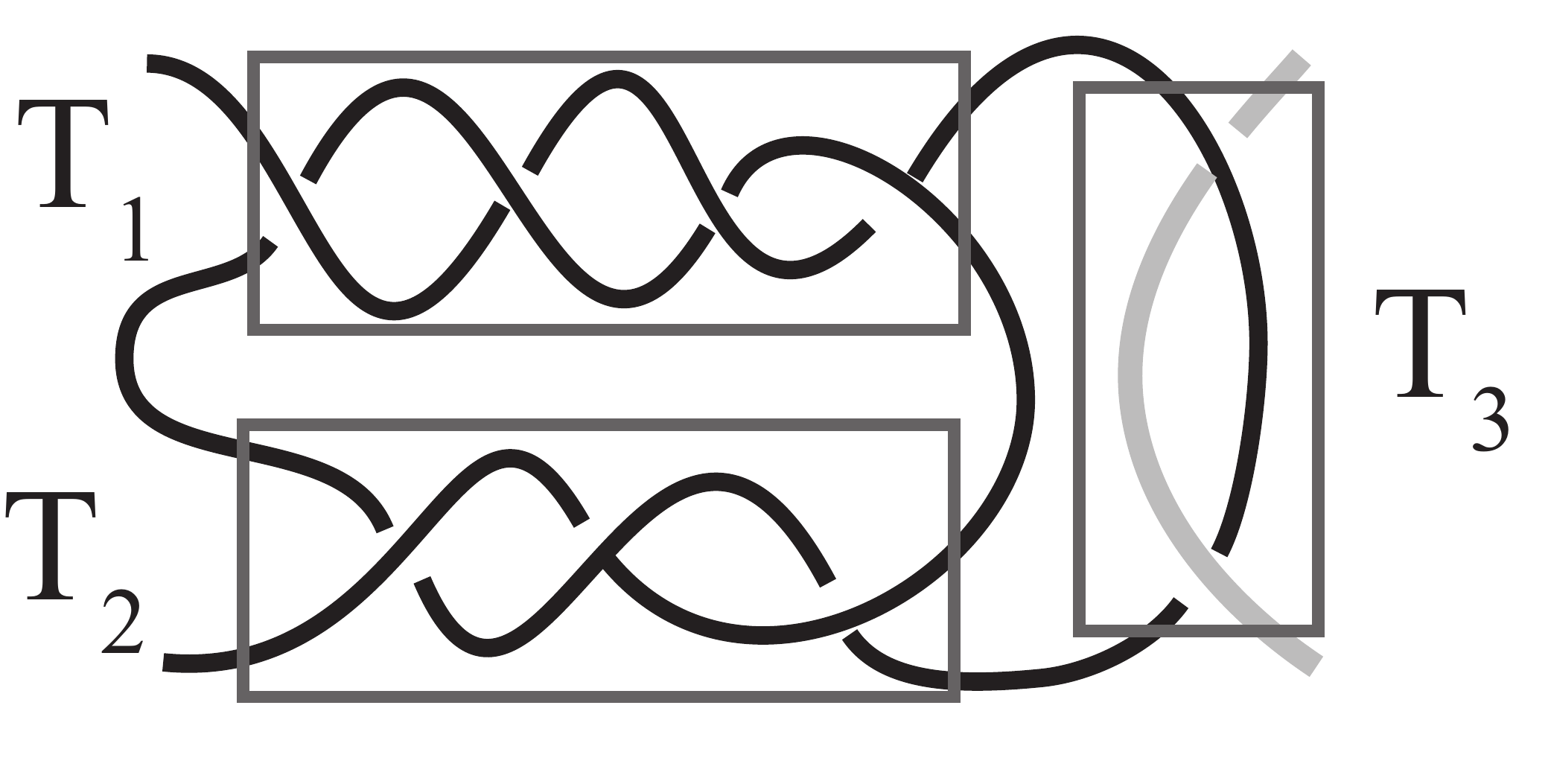}
\caption{A counterexample to the converse of Corollary \ref{arborescentclosure}.}
\label{arbravel}
\end{center}
\end{figure}
 In particular, the grey strand does not pass through $T_1$ or $T_2$.  Observe that the vertex closure $V(T)$ contains no non-trivial knots or links.  However, since $T$ is non-rational, it follows from Lemma \ref{rat0lem} that $V(T)$ is non-planar.  Thus $V(T)$ is a ravel.

\bigskip

\section{Vertex closure with crossing replacement}\label{exception}

We are now interested in whether we can obtain a ravel from a projection of a Montesinos tangle by replacing some number of crossings by vertices and taking the vertex closure. In this case, we need to specify what types of projections we are considering.

\begin{defn}
A projection of a rational tangle $T$ is said to be in \emph{alternating 3-braid form} if it is alternating and has the form of Figure~\ref{3braid}, where each box $A^i$ consists of some number of horizontal twists, and this number is non-zero for all $i> 1$. 
\end{defn}

\begin{figure}[http]
\begin{center}
\includegraphics[width=8cm]{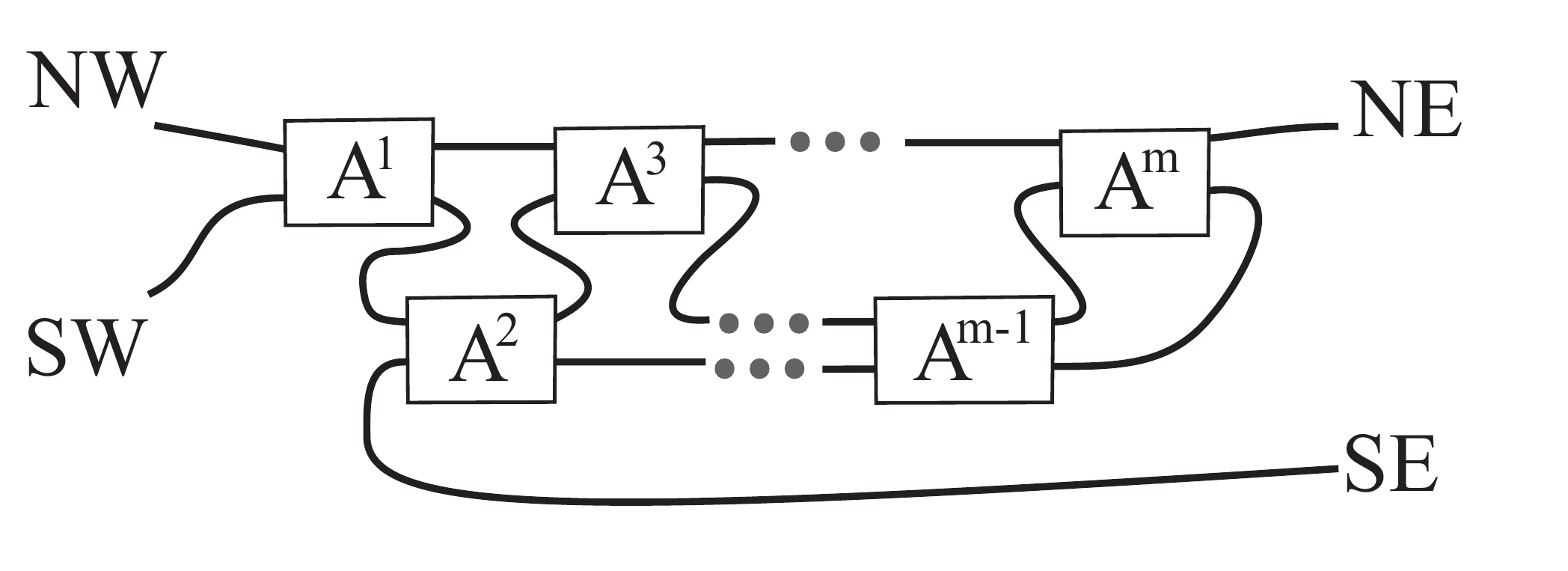}
\caption{The 3-braid form of a rational tangle.}
\label{3braid}
\end{center}
\end{figure}

It follows from Schubert \cite{schubert2} and more recently Kauffman and Lambropoulou \cite{kauff04} that every rational tangle has a projection in alternating 3-braid form.

\begin{defn}
A projection of a Montesinos tangle $T$ is said to be in \emph{standard form} if it is expressed as $T=T_1+ \dots +T_n$, where each $T_i$ is a non-trivial rational tangle in alternating 3-braid form and $n$ is minimal. 
\end{defn}

 Note that every Montesinos tangle with no trivial vertical tangle as a summand has a projection in standard form, though this projection may not be alternating.

\begin{defn}  Let $T$ be a projection of a knot, link, or tangle.  The embedded graph obtained from $T$ by replacing some number of crossings by vertices of valence 4 is denoted by $T'$ and referred to as an \emph{insertion of vertices} into $T$.
\end{defn}

Note that if $T$ is a tangle then $T'$ is not technically an embedded graph, because its endpoints are not vertices.   However, for convenience we will abuse notation and refer to $T'$ as an embedded graph.  Now let $T=T_1+ \dots + T_n$ be a projection of a Montesinos tangle in standard form. Then $T_i'$ denotes the subgraph of $T'$ obtained from $T_i$ by vertex insertion, and $(A_i^j)'$ denotes the subgraph obtained from the $j$th box of twists $A_i^j$ in $T_i$ by vertex insertion.  If there are no vertices in $T_i'$ or in $(A_i^j)'$, then we write $T_i'=T_i$ or $(A_i^j)'=A_i^j$, respectively.  

The following result shows that a ravel cannot occur in the special case where a single crossing is replaced by a vertex and $V(T')$ is a $\theta_4$ graph (i.e., the graph consists of two vertices and four edges between them).

\begin{thm} \label{onecase}{\bf(Farkas, Flapan, Sullivan \cite{farkas11})}  Let $T=T_1+ \dots +T_n$ be a projection of a Montesinos tangle in standard form with $n>1$, and let $T'$ be obtained from $T$ by replacing a single crossing by a vertex such that the vertex closure $V(T')$ is a $\theta_4$ graph.  Then $V(T')$ contains a non-trivial knot and hence is not a ravel. \end{thm}

To see the necessity of the hypothesis that $V(T')$ is a $\theta_4$ graph consider the Montesinos tangle in standard form illustrated on the left in Figure~\ref{example}.  By replacing a crossing in $T_2$ with a vertex and taking the vertex closure as illustrated on the right, we obtain a ravel which is not a $\theta_4$ graph.

\begin{figure}[http]
\begin{center}
\includegraphics[width=10cm]{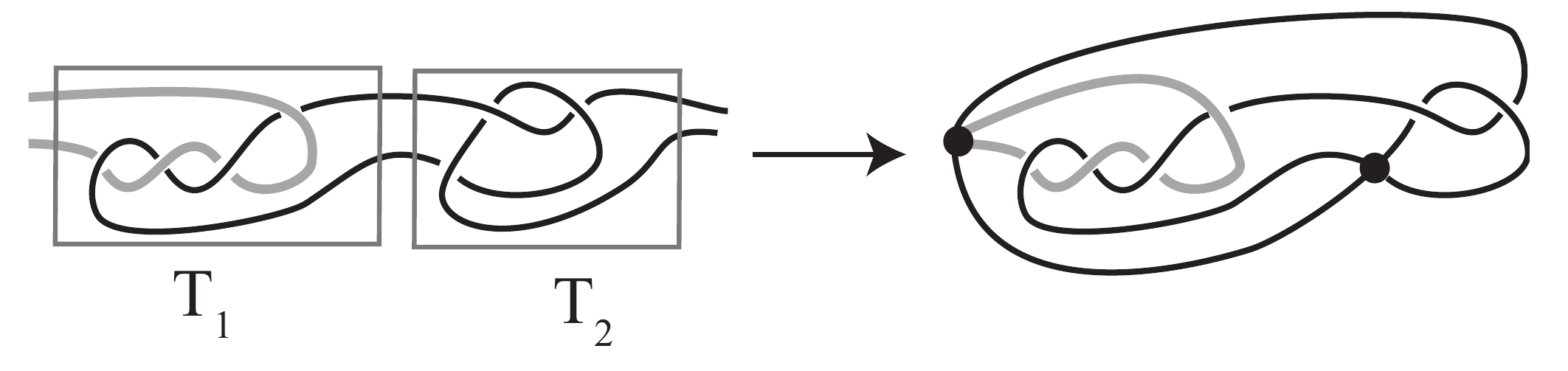}
\caption{A Montesinos tangle which becomes a ravel by inserting one vertex and taking the vertex closure.}
\label{example}
\end{center}
\end{figure}

We will now consider the case where we replace any number of crossings of a Montesinos tangle in standard form by vertices.  We begin with some technical definitions.

\begin{defn}  Let $T$ be a projection of a rational tangle in alternating 3-braid form with boxes $A^1$, \dots, $A^m$ as illustrated in Figure~\ref{3braid}.  A vertex or crossing $x$ of $T'$ is said to be to the \emph{right of} a vertex or crossing $y$ if either $x$ and $y$ are in the same box $(A^i)'$ and $y$ is to the right of $x$ in $(A^i)'$, or $x$ is in the box $(A^i)'$ and $y$ is in the box $(A^j)'$ and $i>j$.
\end{defn}

Observe that given a rational tangle $T$ in 3-braid form, a subtangle $R$ of $T$ containing consecutive boxes $A^j, \dots, A^m$ (illustrated in Figure~\ref{subtangle}) is itself a rational tangle in 3-braid form.

 \begin{figure}[http]
\begin{center}
\includegraphics[width=8cm]{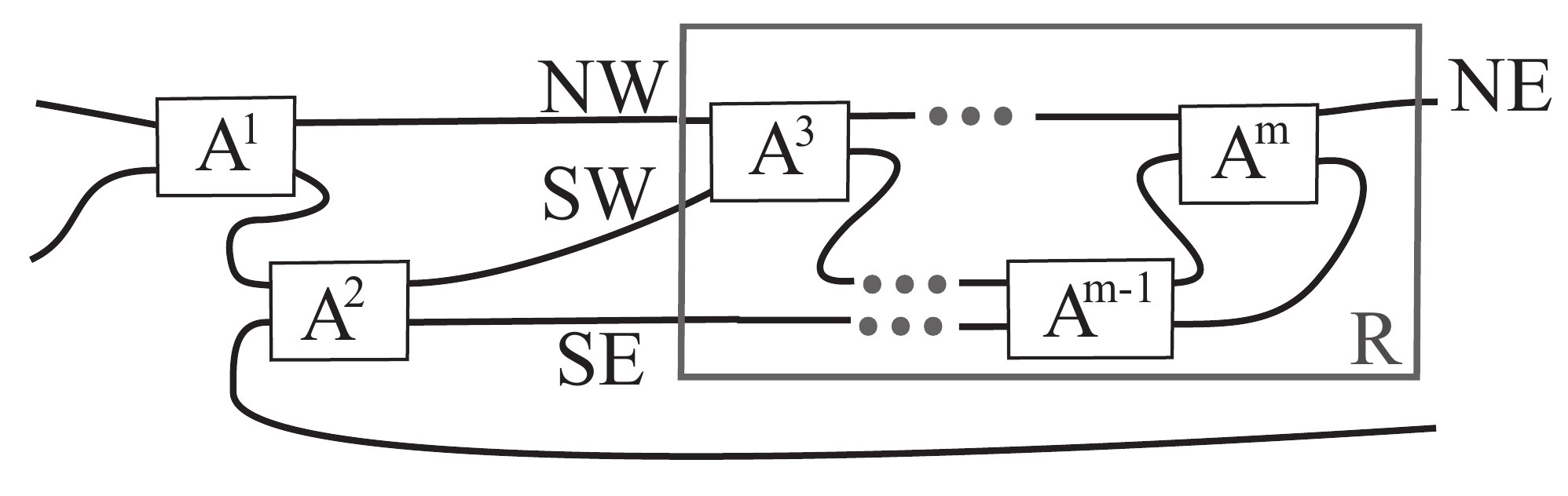}
\caption{The subtangle $R$ is itself a rational tangle in 3-braid form.}
\label{subtangle}
\end{center}
\end{figure}

\begin{defn}\label{exceptional}
Let $T=T_1+ \dots +T_n$ be a projection of a Montesinos tangle in standard form, and let $T'$ be obtained by replacing some nonzero number of crossings by vertices. 
Then $T'$ is said to be an \emph{exceptional vertex insertion} if all of the following conditions hold.  

\begin{enumerate}
\item  There exists precisely one $T_j$ with $\infty$-parity, and $T_j'$ has no vertices. 

\smallskip

\item  For all $k \neq j $, $T_k'$ contains exactly one vertex $v_k$, and $v_k$ is in $(A_k^2)'$ or possibly in $(A_k^3)'$ if $A_k^2$ has a single crossing.

\smallskip

\item For all $k\not =j$, the subtangle $R_k\subseteq T_k$ containing the boxes of $T_k$ to the right of the vertex $v_k$ has at least two crossings.

\smallskip

\item For all $k\not =j$, $T_k'$ has a loop containing $v_k$.

\end{enumerate}
\end{defn}
\medskip

We see as follows that the insertion of vertices illustrated for the tangle in Figure~\ref{example} is exceptional.

\begin{enumerate}
\item $T_1$ is the only $T_i$ with $\infty$-parity, and $T_1'$ has no vertices.
 \smallskip

\item $T_2'$ contains exactly one vertex, and it is in the second box of $T_2'$. (Note that the first box of $T_2'$ has zero crossings). 

\smallskip

\item The subtangle $R_2\subseteq T_2$ has two crossings.

\smallskip

\item  $T_2'$ has a loop containing its vertex.

\end{enumerate}

\medskip

Figure~\ref{exceptionaldiag} illustrates a generalization of the exceptional vertex insertion in Figure~\ref{example}.  Here $T_3$ is any tangle with $\infty$-parity; for each $k\not =3$, $T_k'$ contains exactly one vertex and it replaces the only crossing in $A^2_k$; and $R_k$ is any rational tangle containing at least two crossings such that $T_k'$ has a loop containing $v_k$.  As in Figure~\ref{example}, we obtain a ravel by taking the vertex closure of this exceptional vertex insertion.

 \begin{figure}[http]
\begin{center}
\includegraphics[width=10cm]{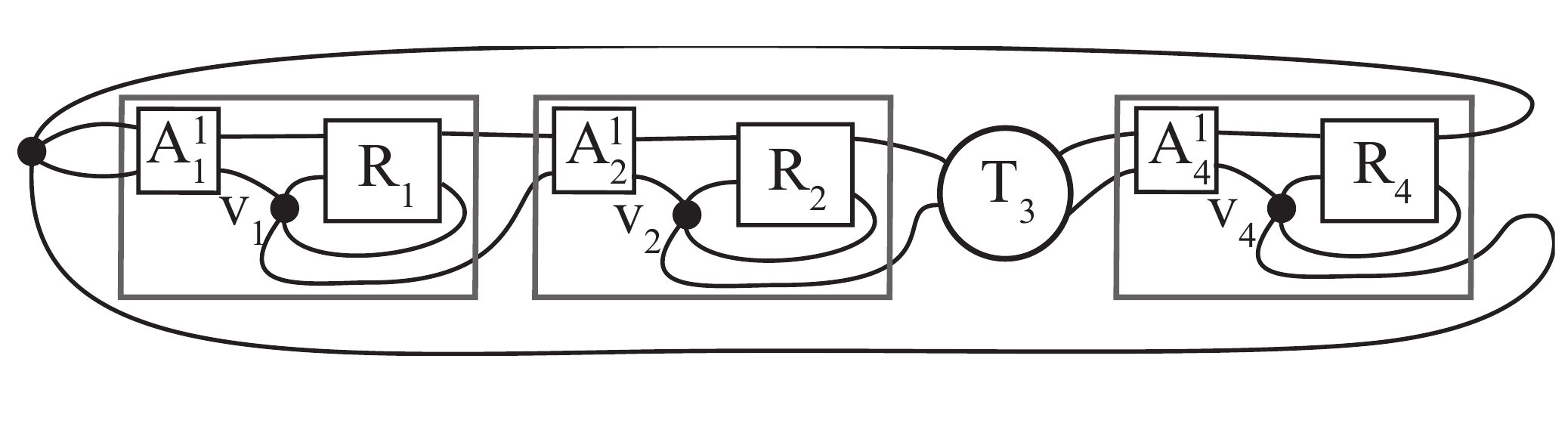}
\caption{A ravel obtained by an exceptional vertex insertion together with vertex closure.}
\label{exceptionaldiag}
\end{center}
\end{figure}

The remainder of the paper is devoted to proving the following theorem.  

\begin{thm}\label{monthm} 
Let $T=T_1+\dots +T_n$ be a projection of a Montesinos tangle in standard form, and let $T'$ be obtained from $T$ by replacing at least one crossing by a vertex. Then the vertex closure $V(T')$ is a ravel if and only if $T'$ is an exceptional vertex insertion. 
\end{thm}

Observe that requirement (4) of an exceptional vertex insertion implies that if $T'$ is an exceptional vertex instertion, then $V(T')$ cannot be a $\theta_4$ graph.  Thus Theorem \ref{monthm} is a generalization of Theorem~\ref{onecase}.  Theorem \ref{monthm} immediately implies the following more simply stated corollary.

\begin{cor}\label{noinfty}
Let $T=T_1+\dots +T_n$ be a projection of a Montesinos tangle in standard form, and let $T'$ be obtained from $T$ by replacing at least one crossing by a vertex. If $V(T')$ is a ravel, then precisely one $T_i$ has $\infty$-parity. 
\end{cor}

Observe that if $T=T_1 + \dots +T_n$ is a projection of a non-rational Montesinos tangle in standard form, then no $T_i$ is horizontal since otherwise $n$ would not be minimal.  Also, by the definition of standard form, no $T_i$ is a trivial  tangle.  Finally, because every vertex in $V(T')$ has valence 4, no arc in $T'$ is forced to terminate at a vertex.  This means that any arc in a $T_i'$ can be extended to go from one of the points NE, SE, NW, SW of $T_i'$ to another.  We will make use of these facts together with the following simplifying assumptions that allow us to remove unnecessary crossings in any $T_i'$.
\newpage

\noindent{\bf Simplifying Assumptions}

\begin{enumerate} 
\item If there are vertices in some box $(A_i^j)'$ of $T_i'$, then we can untwist about them to remove all of the crossings of $(A_i^j)'$. Thus we assume that there are no crossings in any box containing a vertex. 

\medskip

\item  If there is a single crossing to the right of the rightmost vertex of some $T_i'$, then the crossing can be removed by untwisting about the vertex as illustrated in Figure~\ref{1crossing}.  Thus we assume that there are either zero crossings or at least two crossings to the right of the rightmost vertex in any $T_i'$.

\end{enumerate}

\begin{figure}[http]
\begin{center}
\includegraphics[width=5cm]{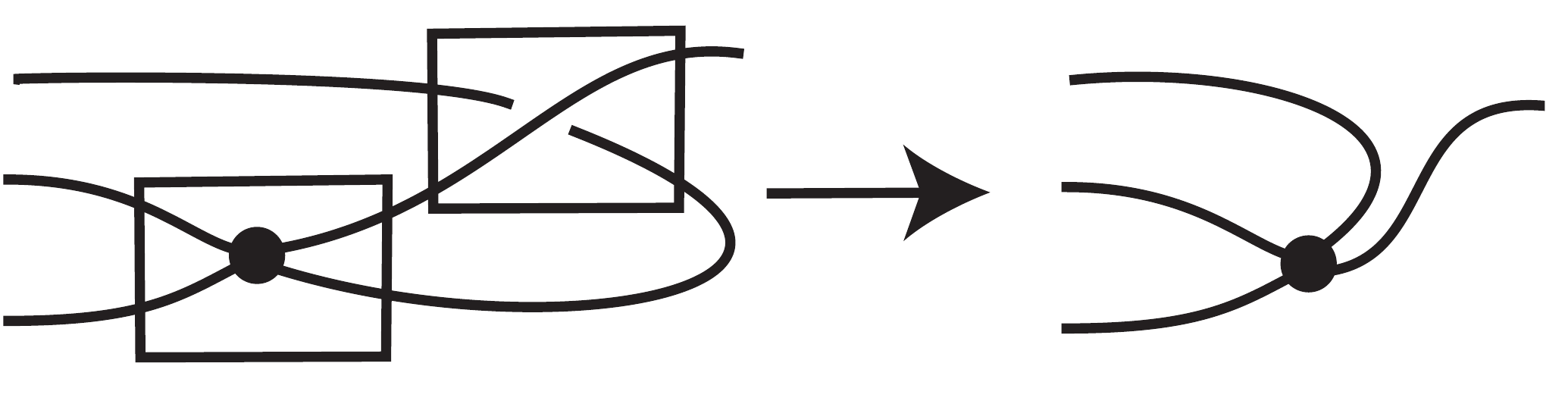}
\caption{A single crossing to the right of the rightmost vertex of $T_i'$ can be removed.}
\label{1crossing}
\end{center}
\end{figure}
\bigskip

The rest of the paper is organized as follows.  In Section~5, we prove two lemmas that we will use to prove the forward direction of Theorem~\ref{monthm}.  We then prove the forward direction in Section~6, and prove the backward direction in Section~7.

\bigskip


\section{Lemmas for the Forward Direction}\label{Lemmas}

\begin{lem}\label{pathfinding}  Let $T$ be a projection of a non-trivial rational tangle in $3$-braid form.  Suppose that $T'$ has at least one vertex and there are no crossings to the right of its rightmost vertex $v^R$. Then for any pair of distinct points $p_1$ and $p_2$ in $T'$, there is a simple path between $p_1$ and $p_2$ in $T'$. 
\end{lem}

\begin{proof} Observe from Figure~\ref{3braid} that if a path $P$ starts at the NE point of $T$ and goes leftward, it will go through the rightmost box $A^m$ precisely once.  In particular, once a path exits from $A^m$, it cannot return to $A^m$.  Thus both strands of $T$ must go through $A^m$.

Recall that for all $i\not =1$, the box $A^i$ contains a non-zero number of crossings.  Since there are no crossings to the right of $v^R$, this means that $v^R$ occurs in the rightmost box $(A^m)'$ of $T'$ as illustrated in Figure~\ref{path}.   Thus both strands of $T$ are involved in the crossing that becomes $v^R$.  

\begin{figure}[http]
\begin{center}
\includegraphics[width=8cm]{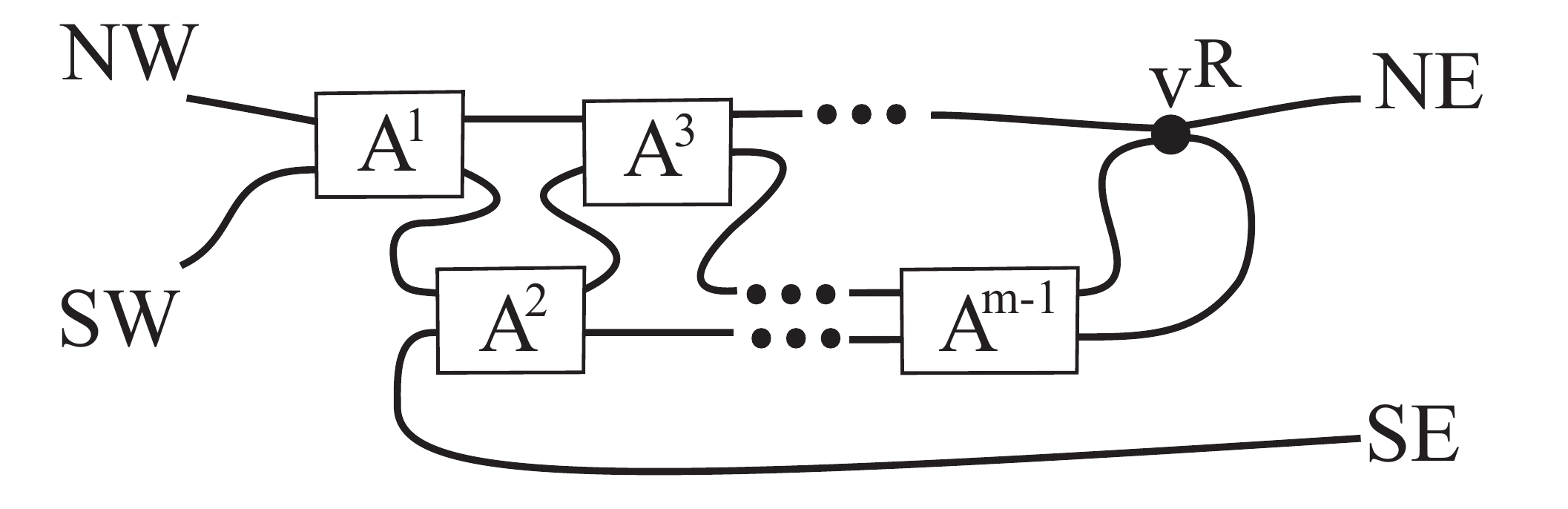}
\caption{Both strands of $T$ are part of the crossing that becomes $v^R$.}
\label{path}
\end{center}
\end{figure}

 Now consider distinct points $p_1$ and $p_2$ in $T'$.  Since both strands of $T$ are part of the crossing that becomes $v^R$, there are paths $P_1$ and $P_2$ in $T'$ joining both $p_1$ and $p_2$ to $v^R$.  Thus $P=P_1\cup P_2$ is a path in $T'$ between $p_1$ and $p_2$.  By removing any loops in $P$ we obtain a simple path joining $p_1$ and $p_2$.  \end{proof}

It follows from Lemma~\ref{pathfinding} that if there are no crossings to the right of the rightmost vertex of $T'$, then there is a simple path in $T'$ between any pair of the NW, SW, NE, SE points of $T'$.  We use Lemma~\ref{pathfinding} to prove our next lemma.

\begin{lem}\label{nofinal} Let $T=T_1 + \dots + T_n$ be a projection of a Montesinos tangle written in standard form, and let $T'$ be obtained from $T$ by replacing at least one crossing by a vertex.  Suppose that every $T_i'$ containing a vertex has at most one crossing to the right of its rightmost vertex $v_i^R$. Then $V(T')$ is not a ravel. 
\end{lem}

\begin{proof}  We assume that $V(T')$ does not contain any non-trivial knots or links, and we will prove that $V(T')$ can be isotoped into the plane.

By Simplifying Assumption (2), we can assume that no $T_i'$ has any crossings to the right of its rightmost vertex $v_i^R$.  Thus any $T_i'$ that contains a vertex must have $v_i^R$ in its rightmost box $(A_i^m)'$.  Now we see in Figure~\ref{examplenovert} that we can remove all of the crossings between $v_i^R$ and the next vertex to its left in $T_i'$, or all of the crossings in $T_i'$ if $v_i^R$ is the only vertex in $T_i'$.  Thus we assume there are no such crossings.

\begin{figure}[http]
\begin{center}
\includegraphics[width=11cm]{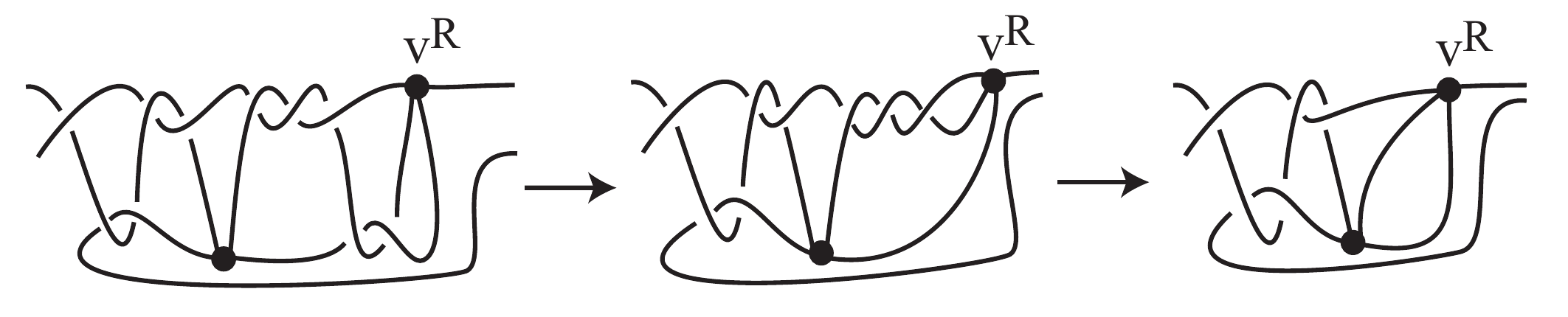}
\caption{We remove all of the crossings between $v_i^R$ and the next vertex to the left in $T_i'$.}
\label{examplenovert}
\end{center}
\end{figure}

We now sequentially prove the following list of claims showing that we can remove all of the crossings of $V(T')$ to obtain a planar embedding.  \medskip

\begin{enumerate}
\item Every $T_i'$ contains a vertex.
\medskip

\item $V(T')$ can be simplified so that there is at most one crossing between any pair of adjacent vertices in each $T_i'$.
\medskip

\item All crossings to the left of the leftmost vertex in each $T_i'$ can be removed.
\medskip

\item  All of the crossings of $V(T')$ can be removed.
\end{enumerate}
\bigskip


\noindent {\bf Claim 1:} Every $T_i'$ contains a vertex.
\medskip

 First we consider the case where $T_1'$ is the only $T_i'$ containing a vertex.  Then there are no crossings to the right of its rightmost vertex, and hence by Lemma~\ref{pathfinding} there is a simple path $L_1$ in $T_1'$ between its NE and SE points.  Now we extend the ends of $L_1$ to the right until either they meet in some $T_i$ with $\infty$-parity or at the closing vertex $w$.  This gives us a simple closed curve $L$.  Since there are no crossings in $T_1$ to the right of its rightmost vertex, $L$ is the connected sum of $D(T_2)$ together with a possibly trivial arc knot to its right.  Now, since $T$ is in standard form, $T_2$ is not horizontal and not a trivial vertical tangle.  Thus by Wolcott's Theorem, $D(T_2)$ is a non-trivial knot or link.  It now follows from Schubert's Theorem that $L$ is a non-trivial knot or link.  As this is contrary to our assumption, this case doesn't occur.

Thus we now assume that for some $i$, $T_{i+1}'$ contains a vertex and $T_i'$ does not.   Let $v_{i+1}^R$ be the rightmost vertex of $T_{i+1}'$.  Since there are no crossings in $T_{i+1}'$ to the right of $v_{i+1}^R$, we can apply Lemma~\ref{pathfinding} to obtain a simple path $L_1$ in $T_{i+1}'$ between its NW and SW points. We will now argue that there is also a simple path between the NW and SW points of $T_i$ whose interior is to the left of $T_i$.

 Suppose no $T_k'$ to the left of $T_i'$ contains a vertex or has $\infty$-parity.  Then there are disjoint simple paths $P_i$ and $Q_i$ going leftwards from the NW and SW points of $T_i$ to the closing vertex $w$.  In this case, $L_2=P_i\cup Q_i$ is a simple path between the NW and SW points of $T_i$ whose interior is to the left of $T_i$.  
 
Thus we assume that either some $T_k'$ to the left of $T_i'$ contains a vertex or some $T_k'$ to the left of $T_i'$ has $\infty$-parity and contains no vertices.  Let $T_k'$ be the closest such subgraph to the left of $T_{i}$.  If $T_k'$ contains a vertex, then there are no crossings to the right of its rightmost vertex, and hence by Lemma~\ref{pathfinding} there is a simple path in $T_k'$ between its NE and SE points.  If $T_k=T_k'$ has $\infty$-parity, then the NE-SE strand is a simple path in $T_k$.  Thus in either case $T_k'$ contains a simple path between its NE and SE points.  By combining this path in $T_k'$ with the strands of all $T_j$ with $k<j<i$ and the arcs between these $T_j$ from $T_k'$ to $T_i'$, we obtain a simple path $L_2$ between the SW and NW points of $T_i$ whose interior is to the left of $T_i$.  Figure~\ref{novert} illustrates the paths $L_1$ and $L_2$ as dotted arcs.

\begin{figure}[http]
\begin{center}
\includegraphics[width=4cm]{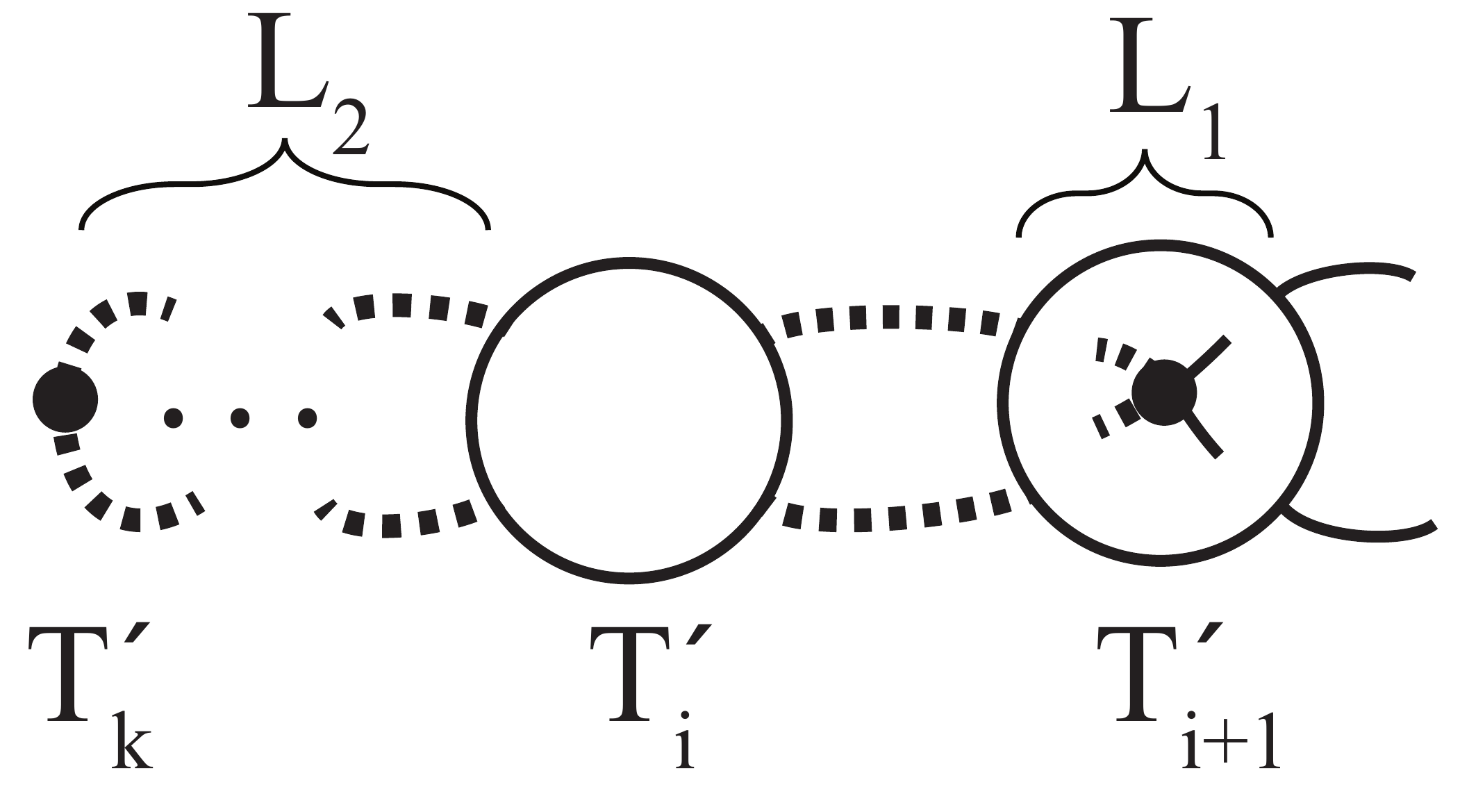}\qquad 
\includegraphics[width=4cm]{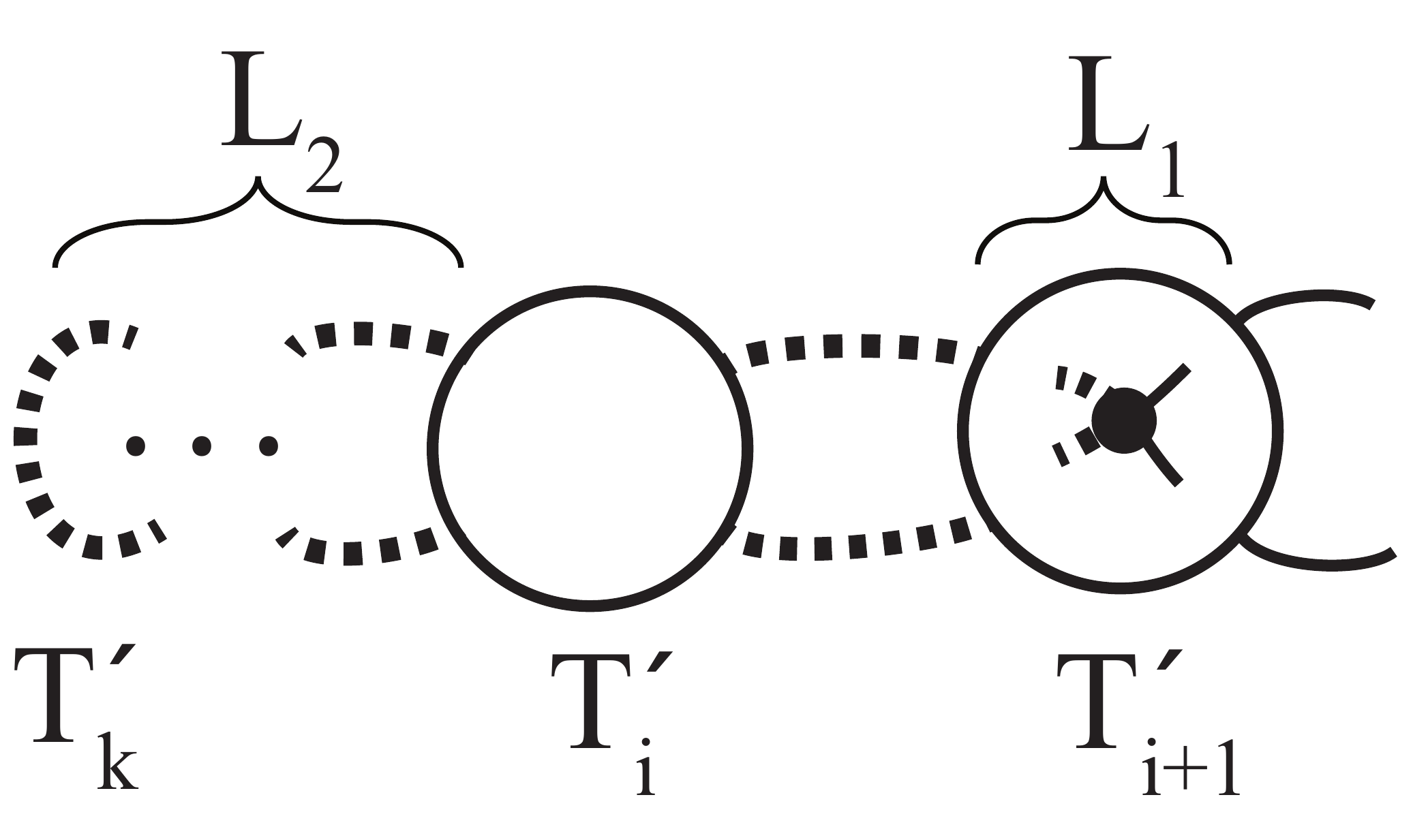}
\caption{On the left, $T_k'$ contains a vertex; and on the right, $T_k'$ has $\infty$-parity.}
\label{novert}
\end{center}
\end{figure}

 Now in any of the above cases, let $L$ denote the arc $L_1$ (in $T_{i+1}'$) together with the arc $L_2$ (to the left of $T_{i+1}'$), as well as the strands of $T_i$, and the arcs joining $T_i'$ and $T_{i+1}'$.  If $T_i$ has $\infty$-parity, then $L$ has two components, and otherwise $L$ has a single component.  In either case, observe that $L$ is the connected sum of some (possibly trivial) knot that lies to the left of $T_i$, together with the denominator closure $D(T_i)$, and some (possibly trivial) knot that lies to the right of $T_i$.  Now, as in the case at the beginning of the claim, $L$ is a non-trivial knot or link.  As this is contrary to our assumption, this proves Claim 1.  
 \medskip
 
 Hence from now on we assume that for every $i$, $T_i'$ contains a vertex, and hence $v_i^R$ is in the rightmost box of $T_i'$. Thus it follows from Lemma~\ref{pathfinding} that every pair of distinct points in any $T_i'$ is joined by a simple path in $T_i'$.\bigskip


 \noindent {\bf Claim 2:} $V(T')$ can be simplified so that there is at most one crossing between any pair of adjacent vertices in each $T_i'$.

 \medskip

Suppose that some $T_i'$ has two or more crossings between a pair of adjacent vertices $v_1$ and $v_2$ contained in boxes $(A_i^j)'$ and $(A_i^k)'$, respectively. 
Note that by Simplifying Assumption (1), we can assume that there are no crossings in any box containing a vertex.  Thus the boxes $(A_i^j)'$ and $(A_i^k)'$ must be distinct. Hence without loss of generality $j < k$.  Thus $v_2$ is to the right of $v_1$.  Also, as we saw at the beginning of the proof, we can assume that there are no crossings between $v_i^R$ and the next vertex to its left.  Hence $v_2$ cannot be the rightmost vertex of $T_i'$.

 Now let $B$ be a ball containing $v_1$ and $v_2$ together with the portion of $T_i'$ between $v_1$ and $v_2$ as illustrated in Figure~\ref{crossingsbetween2};  and let $a$ and $b$ be the points of $\partial B\cap T_i'$ which are separated from $v_1$ and $v_2$ by crossings, as indicated in the figure.  %
\begin{figure}[http]
\begin{center}
\includegraphics[width=8cm]{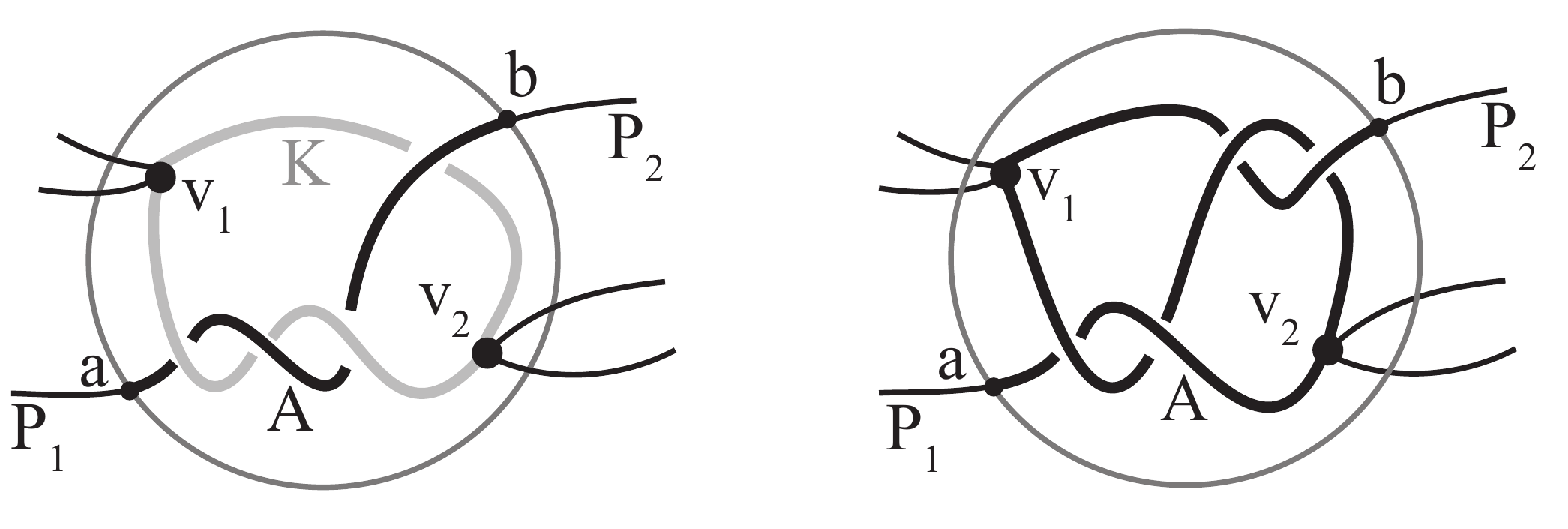}
\caption{The ball $B$ contains the portion of $T_i'$ between $v_1$ and $v_2$.}
\label{crossingsbetween2}
\end{center}
\end{figure}
Note that since $v_1$ and $v_2$ are adjacent vertices, there are no other vertices in $B$.  Thus either $T_i'\cap B$ contains two edges which go between $v_1$ and $v_2$ creating a simple closed curve $K$ and a disjoint arc $A$ going from $a$ to $b$ (as illustrated on the left in  Figure~\ref{crossingsbetween2}), or  $T_i'\cap B$ contains a single arc $A$ joining $a$ and $b$ which goes through both $v_1$ and $v_2$ (as illustrated on the right in Figure~\ref{crossingsbetween2}).

 Now since $v_1$ is contained in the box $(A_i^j)'$, there is a path $P_1$ going leftward from $a$ to the NW, SW, or SE point of $T_i'$ which does not pass through any box $(A_i^t)'$ with $t\geq j$.  Also, there is a path $P_2$ going rightward from $b$ to the rightmost vertex $v_i^R$, and then to the NE point of $T_i'$ which does not pass through any $(A_i^s)'$ with $s\leq k$.  In particular, neither $P_1$ nor $P_2$ contains $v_1$ or $v_2$.

 We will now define a simple closed curve $J$ that contains $P_1\cup P_2\cup A$.  We first do this in the case where  $P_1$ goes from $a$ to the SE point of $T_i'$ as illustrated by the dotted arc in Figure~\ref{C2C1}.  In this case, if $i\not =n$, we extend $P_1$ and $P_2$ rightward to the SW and NW points of $T_{i+1}'$ respectively.  Then by Claim 1 and Lemma~\ref{pathfinding} applied to $T_{i+1}'$, we can join $P_1$ and $P_2$ by a simple path in $T_{i+1}'$.  If $i=n$, then we can extend $P_1$ and $P_2$ rightward until they meet at $w$.  Thus either way we can join $P_1$ and $P_2$ by a simple path.  This give us a simple path $P$ from $a$ to $b$ whose interior is disjoint from $B$.  Now let $J$ denote the simple closed curve $P\cup A$. Then $J$ meets $\partial B$ only in the points $a$ and $b$.

\begin{figure}[http]
\begin{center}
\includegraphics[width=8cm]{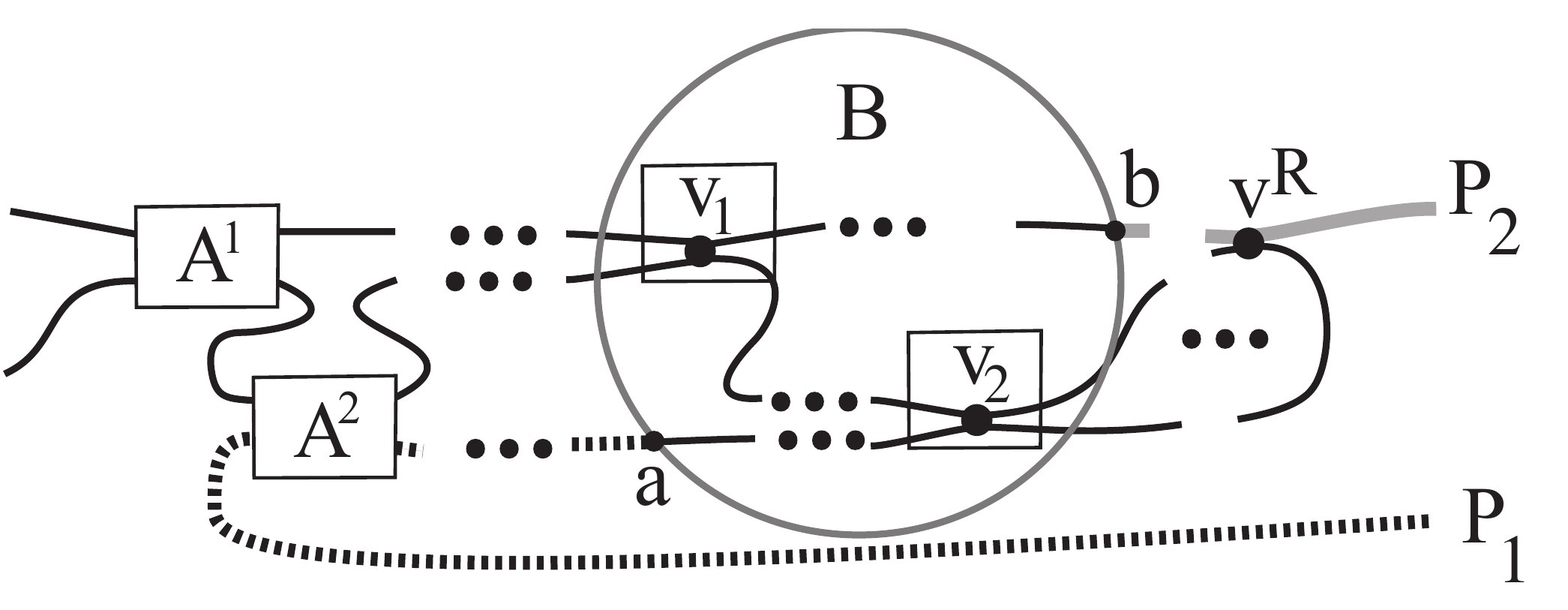}
\caption{There is a simple path $P$ from $a$ to $b$ consisting of $P_1$, $P_2$, and an arc in $T_{i+1}'$.}
\label{C2C1}
\end{center}
\end{figure}

 Next suppose that $P_1$ does not go to the SE point of $T_i'$.  Since $P_1$ does not go to the NE point of $T_i'$,  without loss of generality we can assume $P_1$ goes from $a$ to the SW point of $T_i'$ as illustrated by the dotted arc in Figure~\ref{C2C2}. 
\begin{figure}[http]
\begin{center}
\includegraphics[width=8cm]{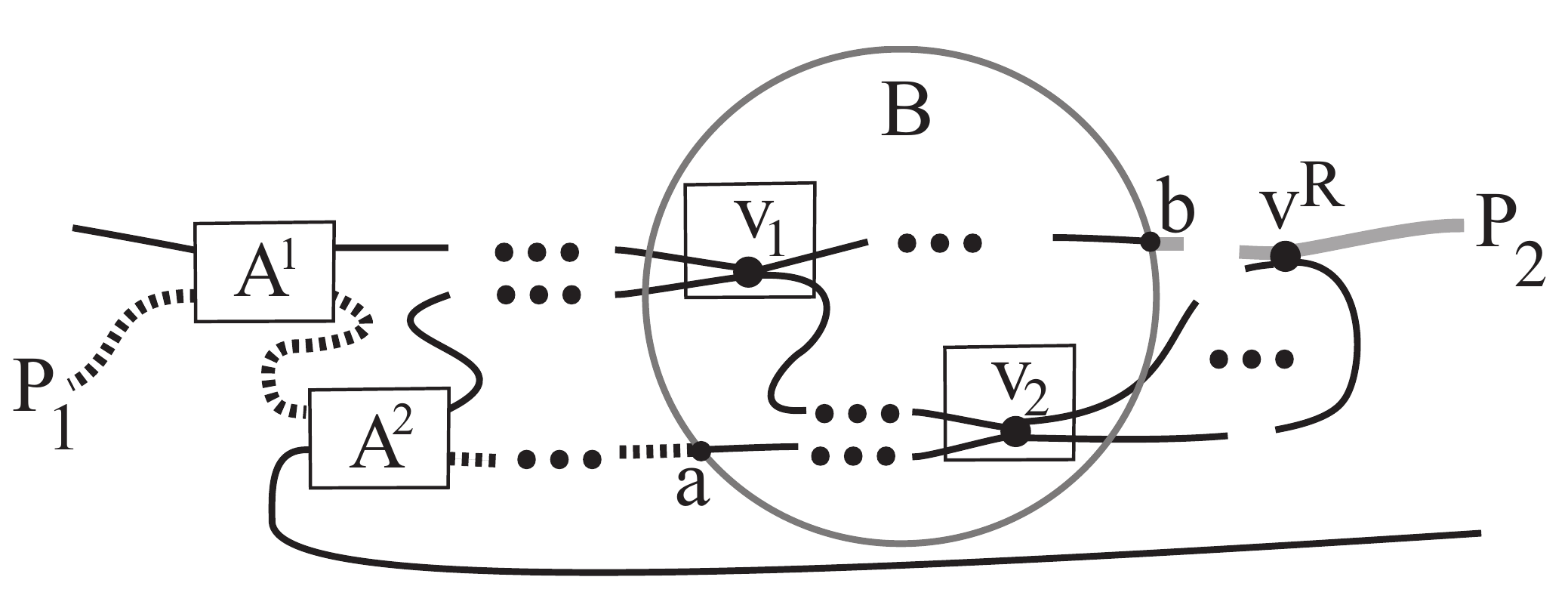}
\caption{We extend $P_1$ leftward and $P_2$ rightward until they meet at the closing vertex $w$.}
\label{C2C2}
\end{center}
\end{figure}
 By Claim 1 and Lemma~\ref{pathfinding}, we can now extend $P_1$ and $P_2$ leftward and rightward respectively until they meet at the closing vertex $w$ giving us a simple path $P$ from $a$ to $b$ whose interior is disjoint from $B$.  Now let $J$ denote the simple closed curve $P\cup A$. Again $J$ meets $\partial B$ only in the points $a$ and $b$.

In either of the above cases, if the arc $A$ contains $v_1$ and $v_2$, we let $L=J$, and otherwise we let $L=J\cup K$.  Now since $T_i'$ has at least two crossings between $v_1$ and $v_2$, the vertex closure $V(L \cap B)$ contains at least two crossings and is reduced and alternating. Thus it follows from Thistlethwaite's Theorem and Schubert's Theorem that $L$ is a non-trivial knot or link in $V(T')$.  As this is contrary to our assumption, we have proven Claim~2.  

Hence from now on we assume that there is at most one crossing between any pair of adjacent vertices in each $T_i'$.
\bigskip


\noindent {\bf Claim 3:} All crossings to the left of the leftmost vertex in each $T_i'$ can be removed.

\medskip

If $T$ is rational, then we can remove all of the crossings to the left of the leftmost vertex $v^L$ by untwisting about the closing vertex $w$ from left to right, as illustrated in Figure~\ref{ratcaseuntwist}.  Thus for the rest of the proof of this claim we assume that $T$ is not rational, and hence $n>1$.

\begin{figure}[http]
\begin{center}
\includegraphics[width=11cm]{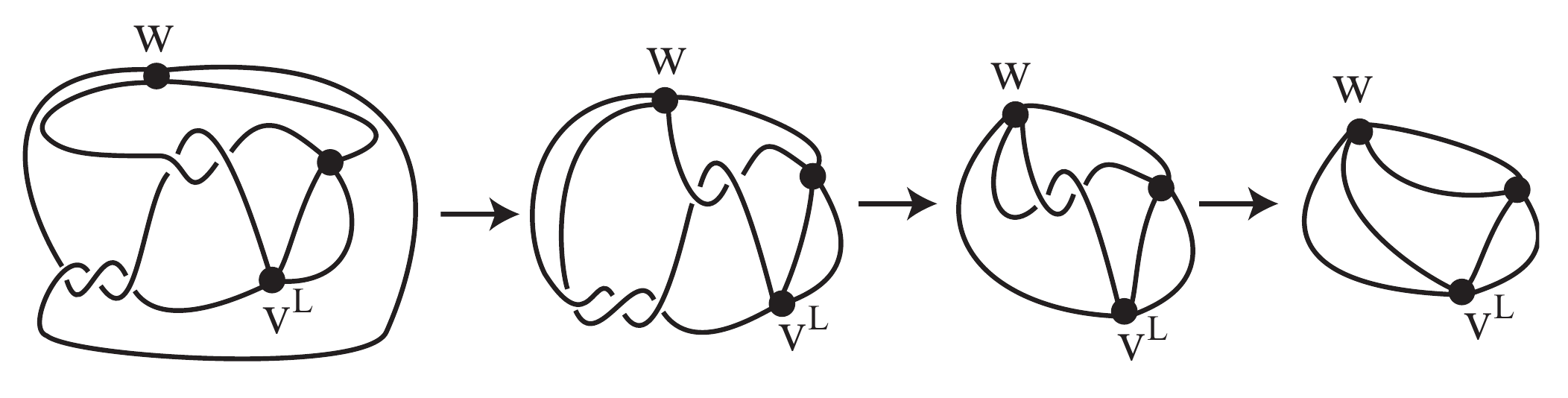}
\caption{If $T$ is rational, we can remove all of the crossings to the left of $v^L$. }
\label{ratcaseuntwist}
\end{center}
\end{figure}

Let $T_i'$ be the leftmost subgraph of $T'$ that contains at least one crossing to the left of its leftmost vertex $v_i^L$.   By Simplifying Assumption (1), there are no crossings in the box with $v_i^L$.  Thus $(A_i^1)'$ cannot contain any vertices, and hence $A_i^1=(A_i^1)'$.   %
\begin{figure}[http]
\begin{center}
\includegraphics[width=11cm]{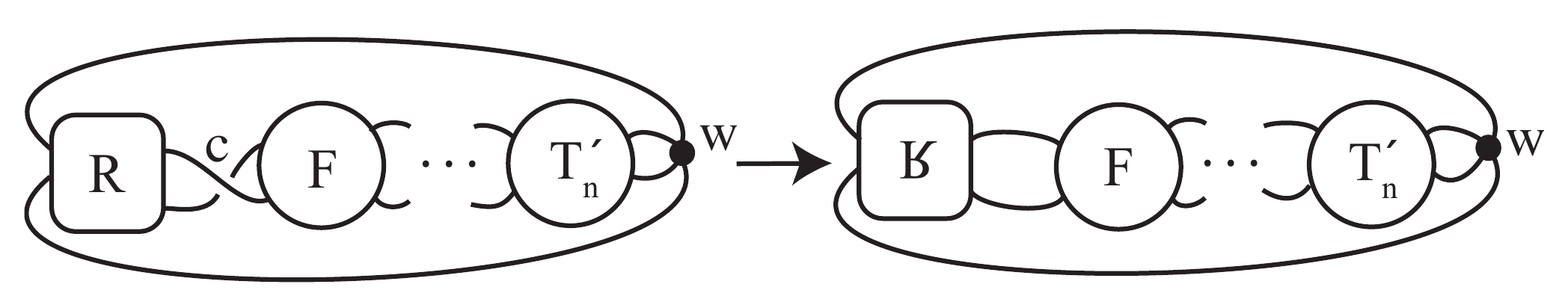}
\caption{Removing a crossing of $A_i'$. }
\label{case1}
\end{center}
\end{figure}
Suppose that $A_i^1$ contains a crossing. Let $c$  denote the leftmost crossing of $A_i^1$; let $R$ denote a ball whose intersection with $V(T')$ is $(T_1+\dots+T_{i-1})'$; and let $F$ denote the part of $T_i'$ to the right of $c$ (see the left side of Figure~\ref{case1}).

 We flip $R$ over to remove the crossing $c$.  This adds a crossing to the left of $R$ which can be removed by untwisting the strands around the closing vertex $w$.  Thus we get the illustration on the right of Figure~\ref{case1}.

 We repeat this operation until we have removed all of the crossings in $A_i^1$.  This proves the claim in the case where $v_i^L$ is in $(A_i^1)'$ or $(A_i^2)'$.  Thus we assume for the sake of contradiction that $v_i^L$ is not in $(A_i^1)'$ or $(A_i^2)'$.  
 
 Next suppose there is only one crossing  in $T_i'$ to the left of $v_i^L$.  Since every box $(A_i^k)'$ with $k>1$ must either contain a crossing or a vertex, this means that $(A_i^2)'$ contains one crossing and $v_i^L$ is in $(A_i^3)'$.  Hence $(A_i^2)'$ has no vertices and $(A_i^3)'$ has no crossings.  Thus we have the illustration on the left of Figure~\ref{case2}, where the ball $R$ contains the subgraphs $T_1'$, \dots, $T_{i-1}'$ and the ball $F$ contains the boxes $(A_i^j)'$ with $j>3$.  We can now remove the crossing in $A_i^2$ by flipping both $R$ and $F$ and untwisting the strands around $w$ as illustrated on the right side of Figure~\ref{case2}.  Thus we assume there are at least two crossings in $T_i'$ to the left of $v_i^L$.  

\begin{figure}[http] 
\begin{center} 
\includegraphics[width=11cm]{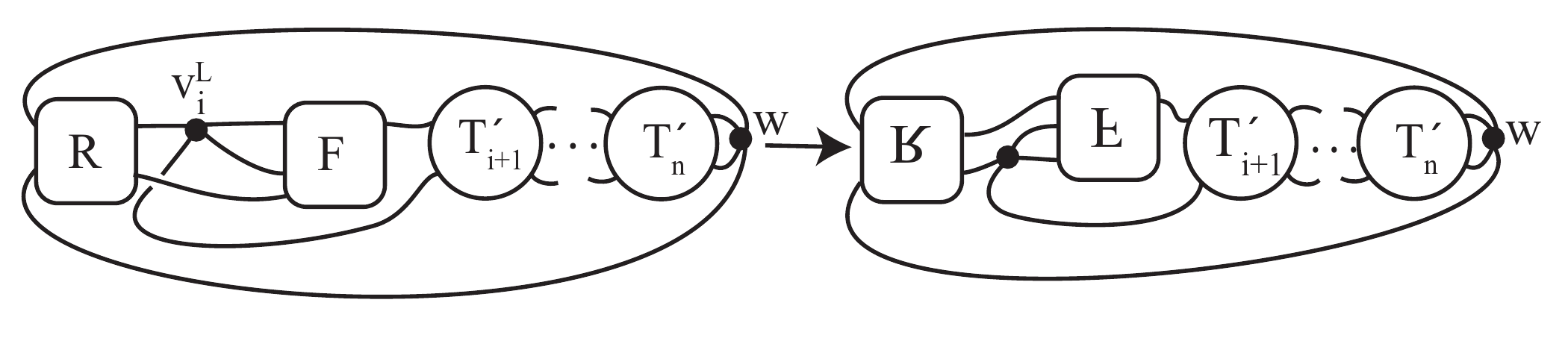}
\caption{Removing a single crossing in $A_i^2$ when $v_i^L$ is in $(A_i^3)'$.}
\label{case2}
\end{center}
\end{figure}

Now let $B$ be a ball containing $v_i^L$ together with the part of $T_i'$ that is to the left of $v_i^L$ as illustrated in Figure~\ref{Claim3Proof}.  
\begin{figure}[http]
\begin{center}
\includegraphics[width=9cm]{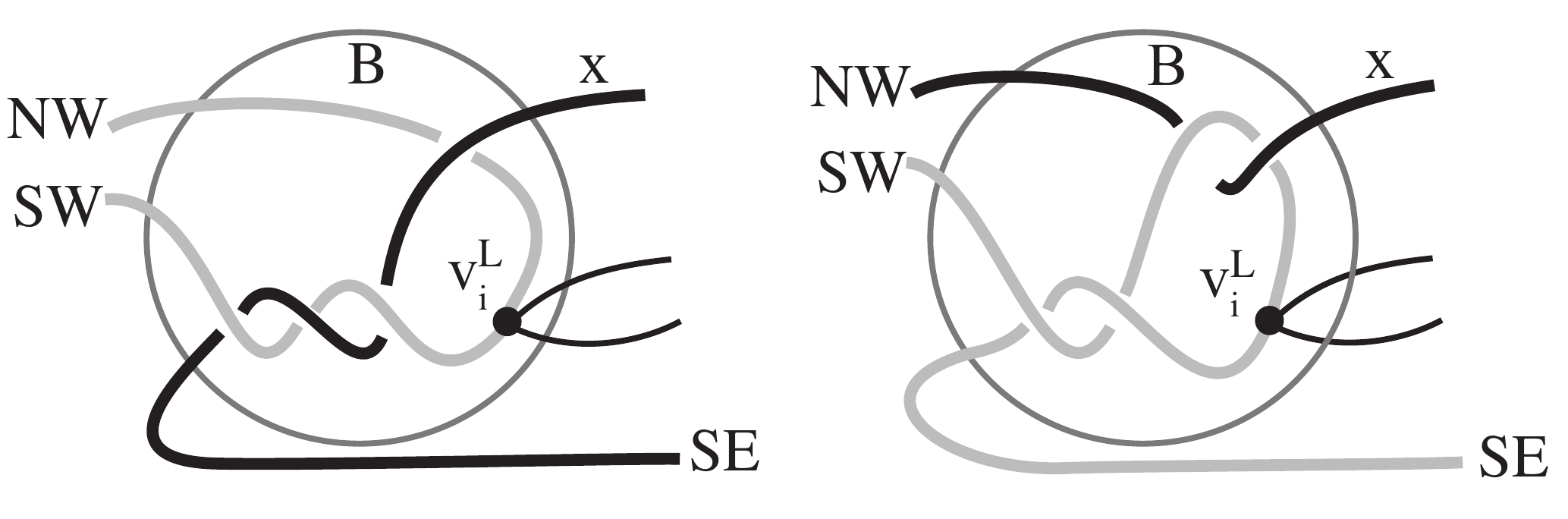}
\caption{We can extend the bold black and grey arcs to get one or two simple closed curves.}
\label{Claim3Proof}
\end{center}
\end{figure}
Since $T_i'$ contains the vertex $v_i^R$ in its rightmost box, the arc marked $x$ can be extended rightward by a simple path to the NE point of $T_{i}'$.  Now the NW and SW points of $T_i'$ are joined together either by an arc through $v_{i-1}^R$ or through $w$ in the case where $i=1$, and the NE and SE points of $T_i'$ are joined together either by an arc through $T_{i+1}'$ or through $w$ if $i=n$.  These arcs, together with the bold black and grey arcs in Figure~\ref{Claim3Proof}, give us one or two simple closed curves which we denote by $L$.  Since $V(L\cap B)$ is reduced and alternating and has at least two crossings, it follows from Thistlethwaite's Theorem and Schubert's Theorem that $L$ is a non-trivial knot or link.  As this contradicts our assumption, this case cannot arise.

Thus we have proven Claim~3.  Hence from now on we assume that there are no crossings to the left of the leftmost vertex in every $T_i'$.

\medskip

It follows from our hypotheses together with Claims 1--3 that the only crossings remaining in $V(T')$ are isolated crossings between two adjacent vertices within a single $T_i'$.  

Let $x$ denote the leftmost crossing in $V(T')$.  Then $x$ is between a pair of vertices $v_a$ and $v_b$ in some $T_i'$. Since there is at most one crossing between any pair of adjacent vertices $v_a$ and $v_b$ are in boxes $(A_i^j)'$ and $(A_i^{j+2})'$ respectively, and the crossing $x$ is in the box $(A_i^{j+1})'=A_i^{j+1}$.  Let $G$ denote a ball around all of the boxes $(A_i^k)'$ with $k<j$ in $T_i'$, and let $F$ denote a ball around all of the boxes $(A_i^k)'$with $k>j+2$ in $T_i'$ (see Figure~\ref{finalmove}).  

\begin{figure}[http]
\begin{center}
\includegraphics[width=10cm]{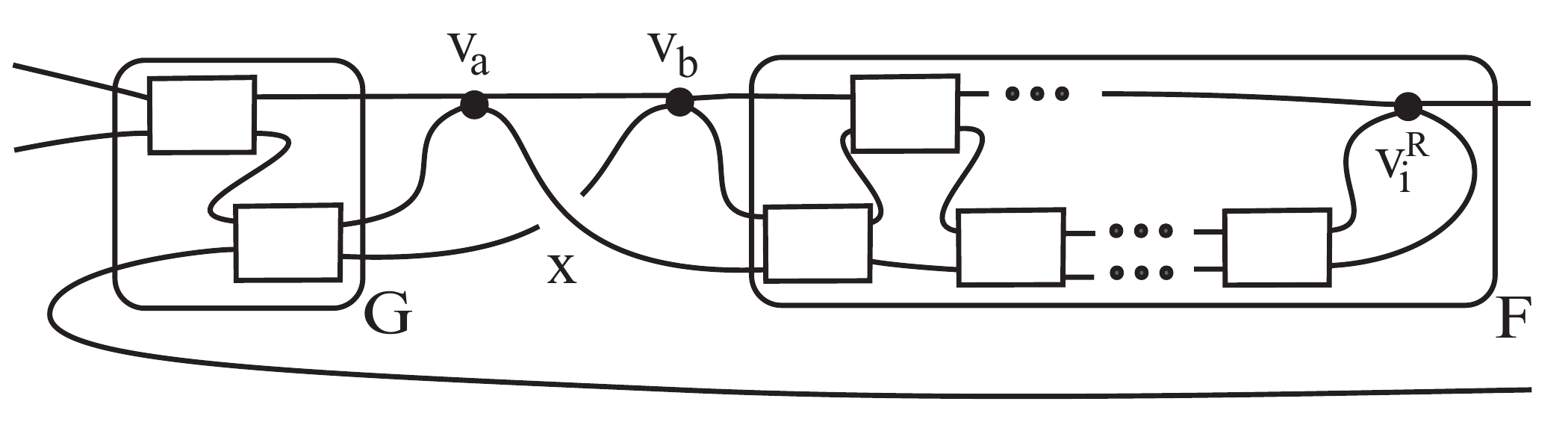}
\caption{$T_i'$ has a single crossing between $v_a$ and $v_b$.}
\label{finalmove}
\end{center}
\end{figure}

Let $B$ denote a ball containing all of the $T_s'$ such that $s<i$.  Then $V(T')$ is the embedded graph illustrated on the top of Figure~\ref{case3}.  Note that $B$ and $G$ contain no crossings since $x$ is the leftmost crossing in $V(T')$.  We now flip $F$ over to remove the crossing $x$ from $V(T')$ as illustrated on the bottom of Figure~\ref{case3}.  We repeat the above argument to sequentially remove all of the remaining crossings in the projection of $V(T')$.  

\medskip

\begin{figure}[http]
\begin{center}
\includegraphics[width=7cm]{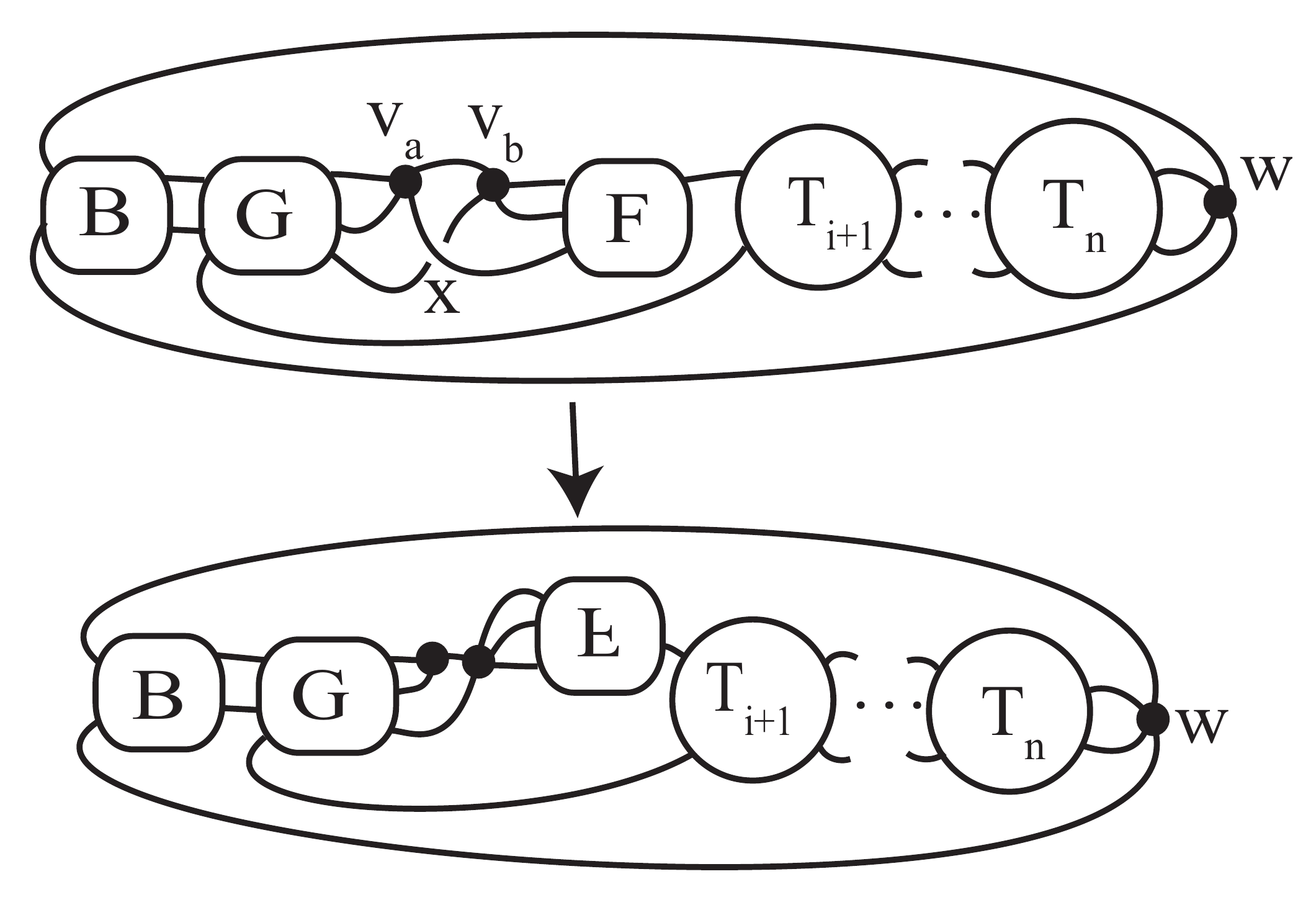}
\caption{We can flip $F$ over to remove the crossing between $v_a$ and $v_b$.}
\label{case3}
\end{center}
\end{figure}

This gives us a planar embedding of $V(T')$.  Thus $V(T')$ is not a ravel.\end{proof}
\bigskip


\section{Proof of the Forward Direction of Theorem~\ref{monthm}}\label{pfforward}

\begin{prop}\label{forward}
Let $T=T_1 + \dots + T_n$ be a projection of a Montesinos tangle in standard form, and let $T'$ be obtained from $T$ by replacing at least one crossing by a vertex.  Suppose that the vertex closure $V(T')$ is a ravel. Then $T'$ is an exceptional vertex insertion.\end{prop}

\begin{proof}  Given any $k$ such that $T_k'$ contains a vertex, let $v_k^R$ denote the rightmost vertex of $T_k'$.  
Then $T_k'$ has the form illustrated in Figure~\ref{vfgeneralform}, where $v_k^R$ may be in the top or the bottom row, and $R_k$ is a ball containing all of the boxes of $T_k'$ that are to the right of $v_k^R$.  

\begin{figure}[http]
\begin{center}
\includegraphics[width=4cm]{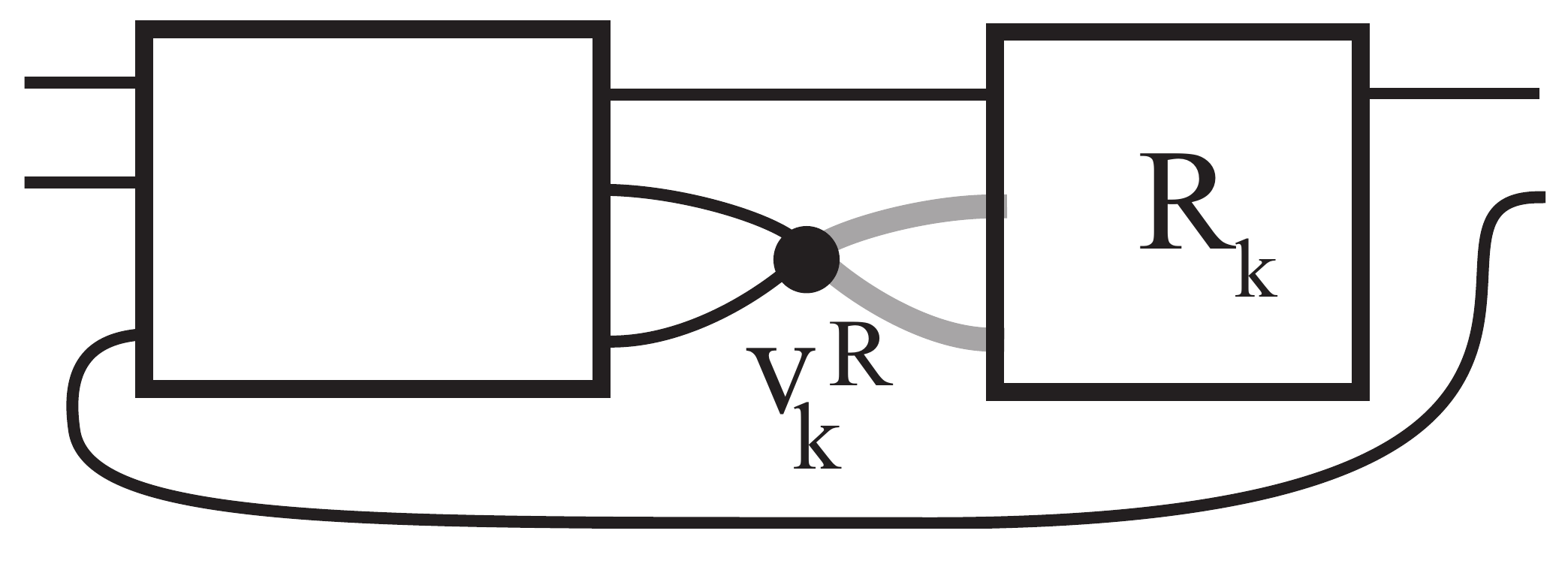}
\caption{$v_k^R$ is the rightmost vertex of $T_k'$, and $R_k$ contains all of the crossings of $T_k'$ that are to the right of $v_k^R$.}
\label{vfgeneralform}
\end{center}
\end{figure}

We now prove that $T'$ is an exceptional vertex insertion by sequentially proving the following list of claims. 
\medskip

\begin{enumerate}
\item  Some $T_j$ has $\infty$-parity, and $T_j'$ has no vertices. 

\medskip

\item For all $k\not =j$, $T_k$ does not have $\infty$-parity, $T_k'$ has at least one vertex, and $R_k$ has at least two crossings. 
\medskip

\item For all $k\not =j$, $T_k'$ contains exactly one vertex $v_k$. 
\medskip

\item  For all $k \neq j $, the vertex $v_k$ is in $(A_k^2)'$ or possibly in $(A_k^3)'$ if $A_k^2$ has a single crossing.
\medskip

\item For all $k \neq j $, the tangle $T_k'$ has a loop containing $v_k$.

\end{enumerate}

\bigskip

We begin by proving Observation 1 which will be used in the proof of Claim 1.  

 \bigskip

\noindent{\bf Observation 1:} If for some $k$, there are at least two crossings in $R_k$, then there is no path in $T_k'$ between its NE and SE point.  Hence $T_k'$ contains paths from its NE and SE points to its NW and SW points.  
\medskip

To prove Observation 1, suppose that there are at least two crossings in $R_k$ and there is a path in $T_k'$ between its NE and SE point.  Then we can extend this path rightward until its ends meet either at a vertex in some $T_j'$, in some $T_j$ with $\infty$-parity, or at the closing vertex $w$. This gives us a simple closed curve $L_1$. If $L_1$ contains $v_k^R$, then all of the crossings of $R_k$ are in $L_1$.  In this case, let $L=L_1$.  Otherwise, as can be seen in  Figure~\ref{vfgeneralform}, there is a grey simple closed curve $L_2\subseteq T_k'$ containing $v_k^R$ such that all of the crossings in $R_k$ are contained in $L_1\cup L_2$. In this case, we let $L=L_1\cup L_2$.

Now let $B_k$ denote the tangle ball for $T_k$.  Then $L$ is the union of $L \cap B_k$ together with an arc outside of $B_k$. Since $L \cap B_k$ contains at least two crossings and is reduced and alternating, by Thistlethwaite's Theorem $L$ is a non-trivial knot or link. However, this contradicts the hypothesis that $V(T')$ is a ravel.  Thus we have proven the observation.
\bigskip

\noindent {\bf Claim 1:}  Some $T_j$ has $\infty$-parity, and $T_j'$ has no vertices. 
\medskip

Since $V(T')$ is a ravel, we know by Lemma~\ref{nofinal} that some $T_k'$ containing a vertex has at least two crossings in $R_k$. Now by Observation~1, there is an arc $P_k$ in $T_k'$ from its NE point to its SW or NW point.  Suppose the endpoints of  $P_k$ can be extended rightward and leftward to $w$, so that we obtain a simple closed curve $L_1$.  If $L_1$ contains $v_k^R$, let $L=L_1$.  Otherwise, there is another simple closed curve $L_2\subseteq T_k'$ containing $v_k^R$ such that all of the crossings in $R_k$ are contained in $L_1\cup L_2$. In this case, we let $L=L_1\cup L_2$.  Now as in the proof of Observation 1, this implies that $L$ is a non-trivial knot or link contradicting the hypothesis that $V(T')$ is a ravel.   Thus $P_k$ cannot be extended so that it passes through every $T_j'$. 

 Thus there must exist some $j$ such that $T_j$ has $\infty$-parity.  Now suppose that $T_j'$ has at least one vertex.  If there are less than two crossings to the right of the rightmost vertex $v_j^R$, then by Simplifying Assumption (2), we can assume there are no crossings in $T_j'$ to the right of $v_j^R$.  Hence by Lemma~\ref{pathfinding}, we could again extend $P_k$ through $T_j$.  On the other hand if there are at least two crossings to the right of $v_j^R$, then by Observation 1, there are paths from the NE and SE points of $T_j'$ to the NW and SW points of $T_j'$.  Thus we could again extend $P_k$ through $T_j'$.  Hence $T_j'$ cannot have any vertices.  Thus we have proven Claim 1.
 \medskip
 
 We now prove Observation~2, which will be used in the proof of Claim 2.
 
 \bigskip

\noindent{\bf Observation 2:}  A single strand of $T_j$ cannot be extended to a simple closed curve in $T'$. 
\medskip

Suppose that some strand of $T_j$ can be extended to a simple closed curve $L_1$ in $T'$.   We now extend ends of the other strand of $T_j$ until they meet at or before $w$.  Since $T_j$ has $\infty$-parity, this gives us a simple closed curve $L_2$ whichis disjoint from $L_1$.  Then $L=L_1\cup L_2$ is the connected sum of $D(T_j)$ and two (possibly trivial) knots.  Since $T_j$ has $\infty$-parity, $D(T_j)$ is a link; and because $T_j$ is non-trivial, by Wolcott's Theorem $D(T_j)$ is non-trivial.  Hence $L$ is also a non-trivial link.  As this contradicts the hypothesis that $V(T')$ is a ravel, Observation 2 follows.

\bigskip

\noindent {\bf Claim 2:}  For all $k\not =j$, $T_k$ does not have $\infty$-parity, $T_k'$ has at least one vertex, and $R_k$ has at least two crossings.

\medskip

Suppose that there is some $k \neq j$ such that $T_k$ has $\infty$-parity. Without loss of generality, $k>j$.  Then we can extend one of the strands of $T_j$ to the right so that the ends meet either in $T_k$ or before.
 As this violates Observation 2, no $T_k$ with $k\not= j$ can have $\infty$-parity. 

Suppose that some $T_k'$ with $k\not =j$ has no vertices.  Without loss of generality, $k>j$.  Since $T_j$ has $\infty$-parity, we can extend the western endpoints of $T_k$ leftward until they meet at or before $T_j$, and we can extend  the eastern endpoints of $T_k$ rightward until they meet at or before $w$.  Let $L$ be the simple closed curve obtained as the union of $T_k$ with these rightward and leftward extensions.  Now $L$ is the connected sum of $D(T_k)$ with (possibly trivial) knots on the right and left.  Recall that since $T$ is in standard form, $T_k$ is not horizontal.  Thus by Wolcott's Theorem $L$ is a non-trivial knot.  As this is contrary to our hypothesis, $T_k'$ must have at least one vertex.

Finally, suppose that some $T_k'$ has at most one crossing in $R_k$.  Then by Simplifying Assumption (2), $T_k'$ has no crossings to the right of $v_k^R$.  Hence by Lemma \ref{pathfinding}, there is a simple path in $T_k'$ between its NW and SW point.  Thus again we can extend one of the strands of $T_j$ to a simple closed curve in $T'$. As this again violates Observation 2, $R_k$ must have at least two crossings.

\bigskip

\noindent {\bf Claim 3:}  For all $k\not =j$, $T_k'$ contains exactly one vertex.

\medskip

Suppose that some $T_k'$ contains at least two vertices.  Without loss of generality $k>j$.  Let $v_k^L$ be the leftmost vertex in $T_k'$.  Then we can illustrate $T_k'$ by Figure~\ref{2vertices}, where all of the crossings of $T_k'$ are in the balls $Q_k$, $S_k$, and $R_k$, and any other vertices of $T_k'$ are contained in $S_k$.  Note that in spite of the way we have illustrated them, $v_k^L$ and $v_k^R$ can each be in either the top or the bottom row.

\begin{figure}[http]
\begin{center}
\includegraphics[width=10cm]{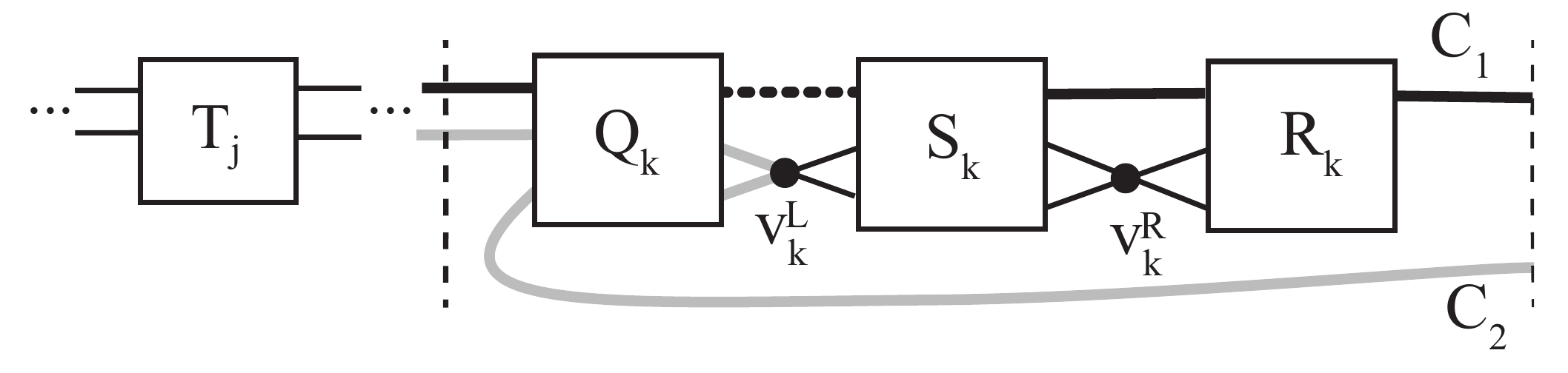}
\caption{$L_1$ intersects $T_k'$ in disjoint arcs $C_1$ and $C_2$.}
\label{2vertices}
\end{center}
\end{figure}

Now we extend the ends of the NE-SE strand of $T_j$ rightward until they meet.  This must occur at the closing vertex $w$ or else it would violate Observation 2.  After removing any loops, we obtain a simple closed curve $L_1$ which intersects $T_k'$ in a pair of disjoint arcs $C_1$ and $C_2$ each going between an Eastern and Western point of $T_k'$.  Without loss of generality, we assume the endpoints of $C_1$ are the NE and NW points of $T_k'$ and the endpoints of $C_2$ are the SE and SW points of $T_k'$ as illustrated in Figure~\ref{2vertices}.

Observe that two arcs of $C_2$ enter $Q_k$ from the left.  Since $Q_k$ has no vertices and does not contain the rightmost box of $T_k$, an arc that enters on the left must leave on the right.  Thus two arcs of $C_2$ must exit $Q_k$ on the right.  Now $C_1$ also exits $Q_k$ on the left, and hence must enter $Q_k$ on the right.   Since $C_1$ and $C_2$ are disjoint, this means that the two arcs entering $v_k^L$ from the left must belong to $C_2$ and the dotted black arc in Figure~\ref{2vertices} belongs to $C_1$.  Furthermore, since $C_2$ does not contain any loops, $C_2$ cannot continue rightward beyond $v_k^L$.  In particular, $v_k^R$ cannot be in $C_2$.

Next suppose that the arc of $C_1$ from $R_k$ to $S_k$ passes through $v_k^R$.  Then $T_k'$ is illustrated in Figure~\ref{ChangeC1}, where the grey dotted arcs entering $S_k$ on the right are connected in some way to the grey dotted arcs exiting $S_k$ on the left.  In this case, there is a path in $T_k'$ going from its NE endpoint passing through both $v_k^L$ and $v_k^R$ and exiting $T_k'$ from its SE endpoint.  However, by Claim 2 we know that $R_k$ contains at least two crossings, and hence by Observation~1 no such path can exist.  Thus the arc of $C_1$ from $R_k$ to $S_k$ cannot pass through $v_k^R$ as it does in Figure~\ref{ChangeC1}.

\begin{figure}[h!]
\begin{center}
\includegraphics[width=10cm]{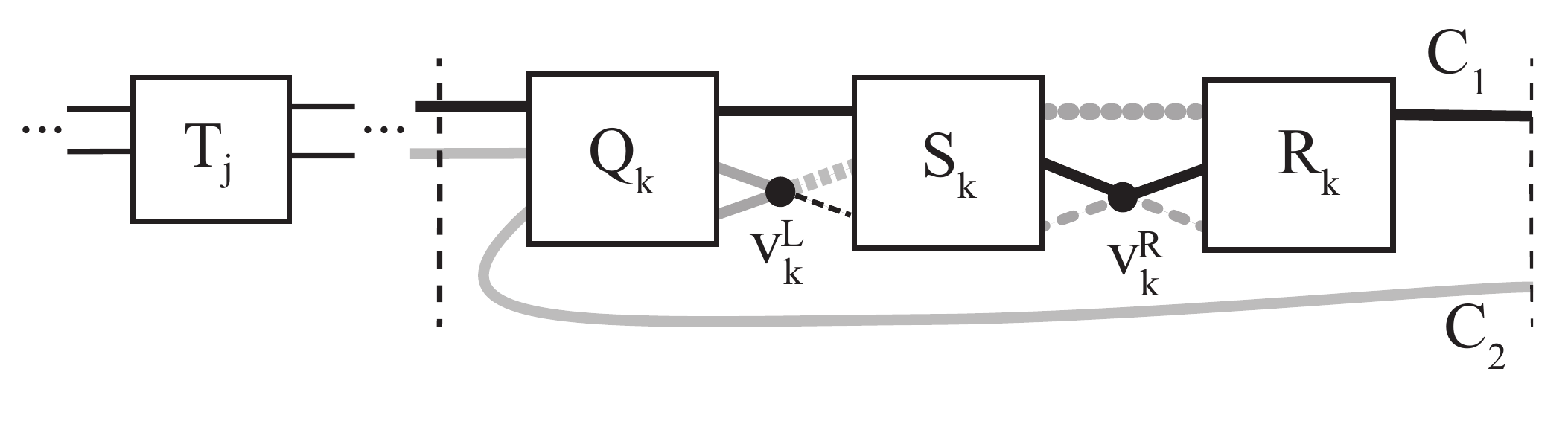}
\caption{$L_1$ intersects $T_k'$ in disjoint arcs $C_1$ and $C_2$.}
\label{ChangeC1}
\end{center}
\end{figure}

Hence either $C_1$ goes through $v_k^R$ and then reenters $R_k$, or $T_k'$ contains a simple closed curve $L_2$ that goes through $v_k^R$ and is disjoint from $C_1$.   In the first case, since $R_k$ is alternating and contains at least two crossings, $L_1$ is a non-trivial knot.  As this is contrary to our assumption, the second case must occur.   However, $R_k$ is itself a rational tangle in alternating $3$-braid form, and by Simplifying Assumption~1 there are no crossings in the same box as $v_k^R$.  Thus there must be at least two crossings between $L_1$ and $L_2$.  But this implies that $L_1\cup L_2$ is a non-trivial link.  As this is again contrary to our hypothesis, $T_k'$ must contain exactly one vertex.


\bigskip

\noindent {\bf Claim 4:} For all $k \neq j $, the vertex $v_k$ is in $(A_k^2)'$ or possibly in $(A_k^3)'$ if $A_k^2$ has a single crossing.

\medskip

If some $T_k'$ has its vertex in the first box, as illustrated in Figure~\ref{nofirst}, then there would be a path in $T_k'$ between its NE and SE points, which would violate Observation 1.

\begin{figure}[http]
\begin{center}
\includegraphics[width=7cm]{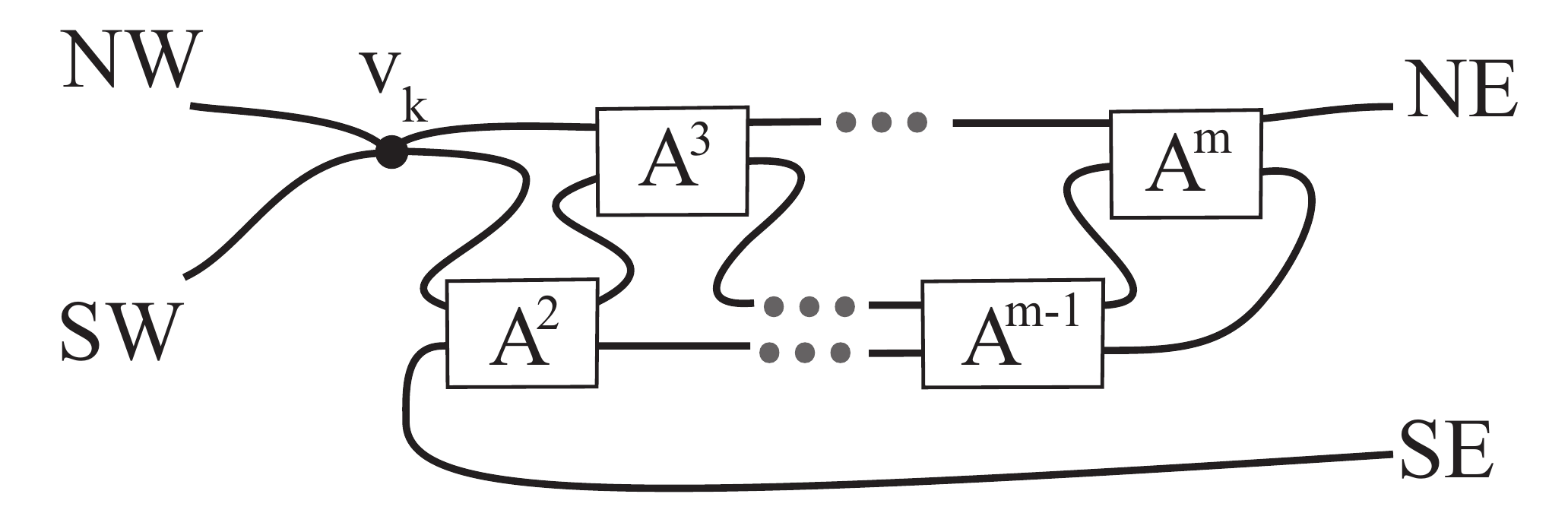}
\caption{If $v_k$ is in $A^1$, then there are paths in $T_k'$ joining the NW and SW points and joining the NE and SE points.}
\label{nofirst}
\end{center}
\end{figure}

Now suppose that some $T_k'$ has its vertex $v_k$ in a box $(A_k^p)'$ where either $p>3$ or $p=3$ and $A_k^2$ has more than one crossing. Let $W_k$ be the tangle consisting of $v_k$ and the part of $T_k'$ to the left of $v_k$, as illustrated in Figure~\ref{looksrational} (though $v_k$ could be in a box to the right of $A^k_3$).  Note that $W_k$ includes the black arcs to the left of $v_k$ but not the grey arcs to the right of $v_k$.  Then $W_k$ is a rational tangle; and since there are at least two crossings in $A_k^2, \dots, A_k^{p-1}$, the tangle $W_k$ is neither a horizontal tangle nor a trivial vertical tangle. 
\begin{figure}[http]
\begin{center}
\includegraphics[width=7cm]{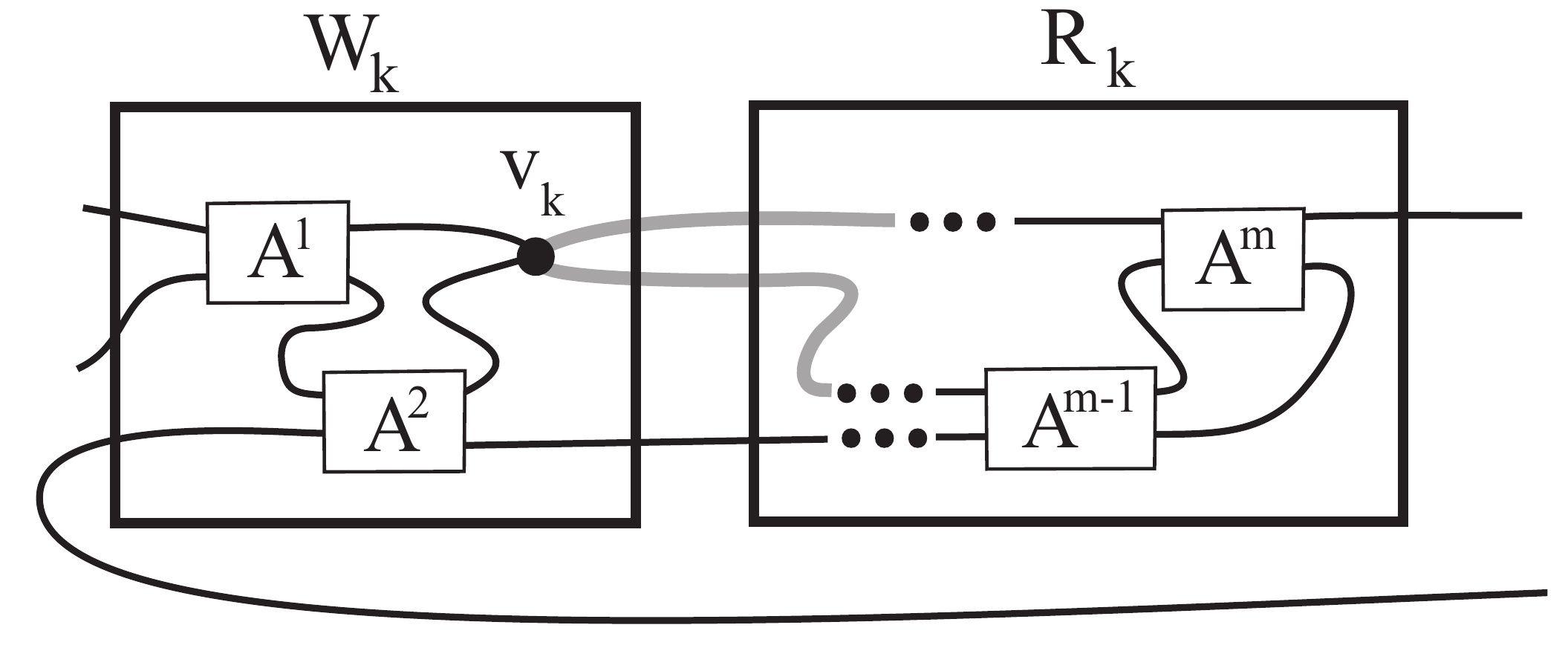}
\caption{$W_k$ is a rational tangle which is neither a horizontal tangle nor a trivial vertical tangle.}
\label{looksrational}
\end{center}
\end{figure}

Now there is a path that goes from the NW and SW points of $W_k$ leftward until its ends meet in $T_j$, at $w$, or at some other vertex.  Also there is a path that goes rightward from the NE and SE points of $W_k$ until the ends meet in $T_j$, at $w$, or at some other vertex.  Note that since $T_j$ has $\infty$-parity, at most one of these paths contains $w$. The union of the two strands of $W_k$ together with these leftward and rightward paths is the connected sum of $D(W_k)$ with two (possibly trivial) knots.  Since $W_k$ is neither horizontal nor a trivial vertical tangle, this connected sum is a non-trivial knot or link.  Thus the vertex $v_k$ must either be in $(A_k^2)'$ or possibly in $(A_k^3)'$ if $A_k^2$ has only one crossing.  
\bigskip

\noindent{\bf Claim 5:} For all $k \neq j $, $T_k'$ has a loop containing $v_k$.
\medskip

It follows from Claim 4 that $T_k'$ has one of the forms illustrated in Figure~\ref{Claim5}.

\begin{figure}[http]
\begin{center}
\includegraphics[width=12cm]{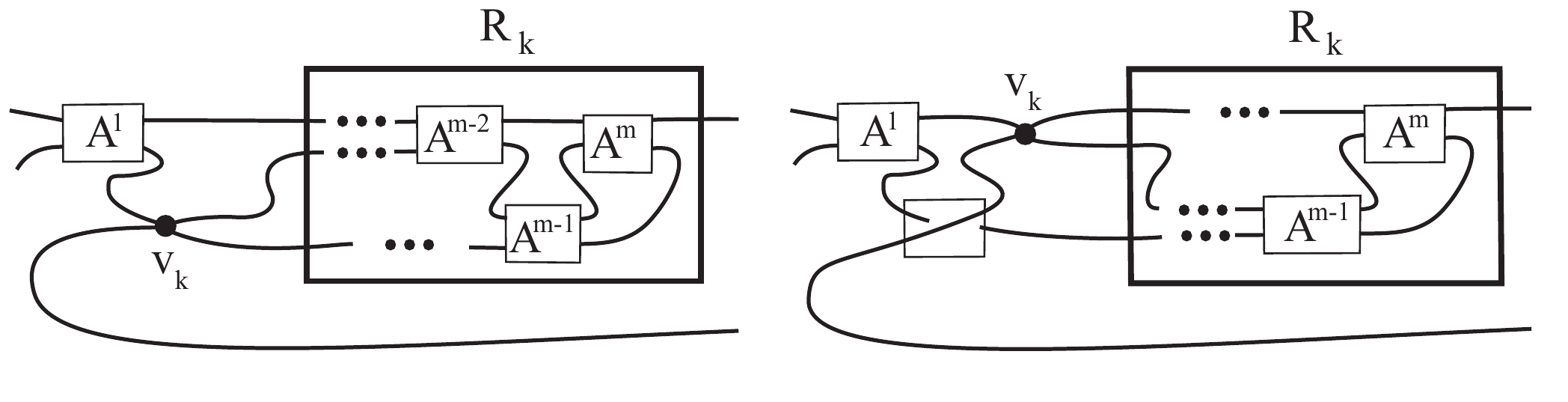}
\caption{$T_k'$ has one of these forms.}
\label{Claim5}
\end{center}
\end{figure}

First we consider the illustration on the left in Figure~\ref{Claim5}.  In this case, if $R_k$ has $0$-parity, then the strands going into $v_k$ from the right are connected together in $R_k$.  Hence they are part of a loop in $T_k'$.  On the other hand, if $R_k$ does not have $0$-parity, then there is a path from the NE point of $R_k$ to $v_k$.  We can then extend this path to get a path in $T_k'$ from its NE point to its SE point.  As this violates Observation 1, this cannot occur.

Next we consider the illustration on the right in Figure~\ref{Claim5}.  Now if $R_k$ has $\infty$-parity, then the strands going into $v_k$ from the right are connected together in $R_k$.  Hence they are part of a loop in $T_k'$.  But if $R_k$ does not have $\infty$-parity, then there is a path from the NE point of $R_k$ to $v_k$. Again we can extend this path to get a path in $T_k'$ from its NE point to its SE point violating Observation 1.  Thus in either case, $T_k'$ has a loop containing $v_k$.
\medskip

Now it follows from Claims 1 through 5 that $T'$ is an exceptional vertex insertion. 
\end{proof}
\bigskip


\section{The Proof of the Backward Direction of Theorem~\ref{monthm}}\label{pfbackward}
\medskip

In order to prove the backward direction of Theorem~\ref{monthm}, we make use of the following definition and theorem due to Sawollek  \cite{sawollek}.
 
\begin{defn}  Let $G$ be a 4-valent graph embedded in $\mathbb{R}^3$.  The set of {\em associated links} $S(G)$ consists of all knots and links that can be obtained from $G$ by replacing a neighborhood of each vertex of $G$ by a rational tangle.
\end{defn}

 \begin{Sawollek} \label{sawollektheorem} \cite{sawollek}
Let $G$ be a 4-valent graph embedded in $\mathbb{R}^3$. The set of associated links $S(G)$ is an isotopy invariant of $G$. 
 \end{Sawollek}

\begin{prop}\label{backward}
Let $T=T_1 + \dots + T_n$ be a projection of a Montesinos tangle in standard form, and suppose that $T'$ is obtained from $T$ by an exceptional vertex insertion. Then the vertex closure $V(T')$ is a ravel.
\end{prop}

\begin{proof}  By the definition of an exceptional vertex insertion, there is a single $T_j$ with $\infty$-parity, and $T_j'$ has no vertices. Without loss of generality we assume that $1<j\leq n$.  Also, for all $k\not =j$, $T_k'$ has a single vertex $v_k$ which is either in $(A_k^2)'$ or possibly in $(A_k^3)'$ if $A_k^2$ has only one crossing.  Furthermore, $R_k$ (the subtangle of $T_k$ consisting of the boxes to the right of $v_k$) is a rational tangle with at least two crossings and $T_k'$ has a loop containing $v_k$.  Now for each $k$ such that $v_k$ is in $(A_k^3)'$, we move $v_k$ to $(A_k^2)'$ by flipping $R_k$ as illustrated in Figure~\ref{onecrossingina2}.

\begin{figure}[http]
\begin{center}
\includegraphics[width=8cm]{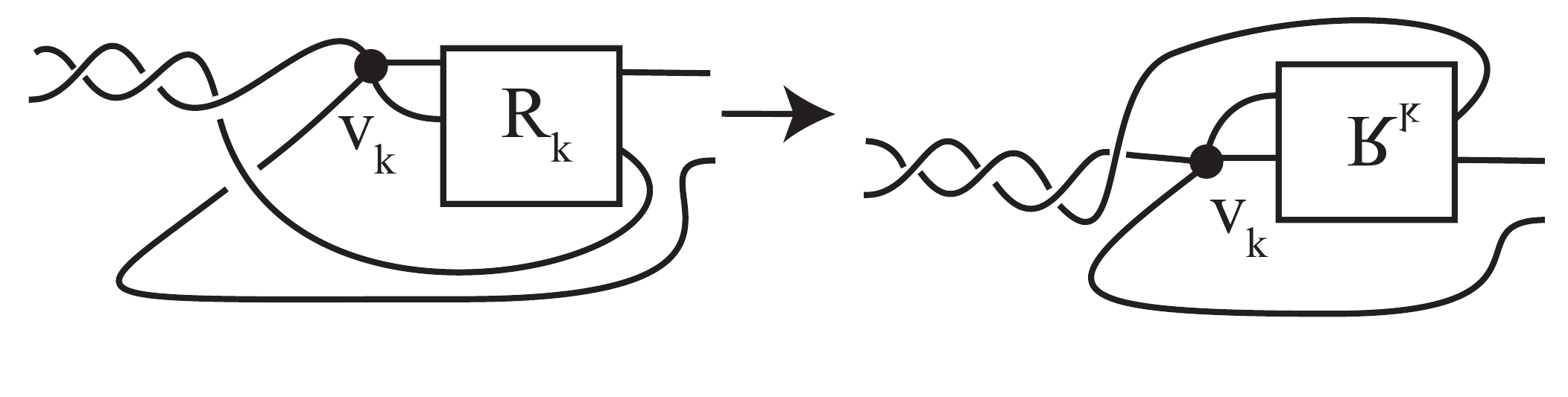}
\caption{When $v_k$ is in $A_k^3$, we flip $R_k$ to move $v_k$ to $(A_k^2)'$.}
\label{onecrossingina2}
\end{center}
\end{figure}

 Next, for each sequential $k>1$ such that $k\not =j$, we flip the part of the projection of $V(T')$ to the left of $A_k^1$ repeatedly to move the crossings of $A_k^1$ to $A_1^1$.  Then we remove all of the accumulated crossings from $A_1^1$ by twisting the strands around $w$. We illustrate this in Figure~\ref{flypeai}, where $A_1^1$ begins with zero crossings and $A_2^1$ begins with three crossings.  In the second picture we have moved the three crossings of $A_2^1$ to $A_1^1$, and in the third picture we have removed all of the crossings from $A_1^1$.  This gives us a projection of $V(T')$ such that for each $k\not =j$, the vertex $v_k$ is in $(A_k^2)'$ and all of the crossings of $T_k'$ are in $R_k$.

\begin{figure}[http]
\begin{center}
\includegraphics[width=10cm]{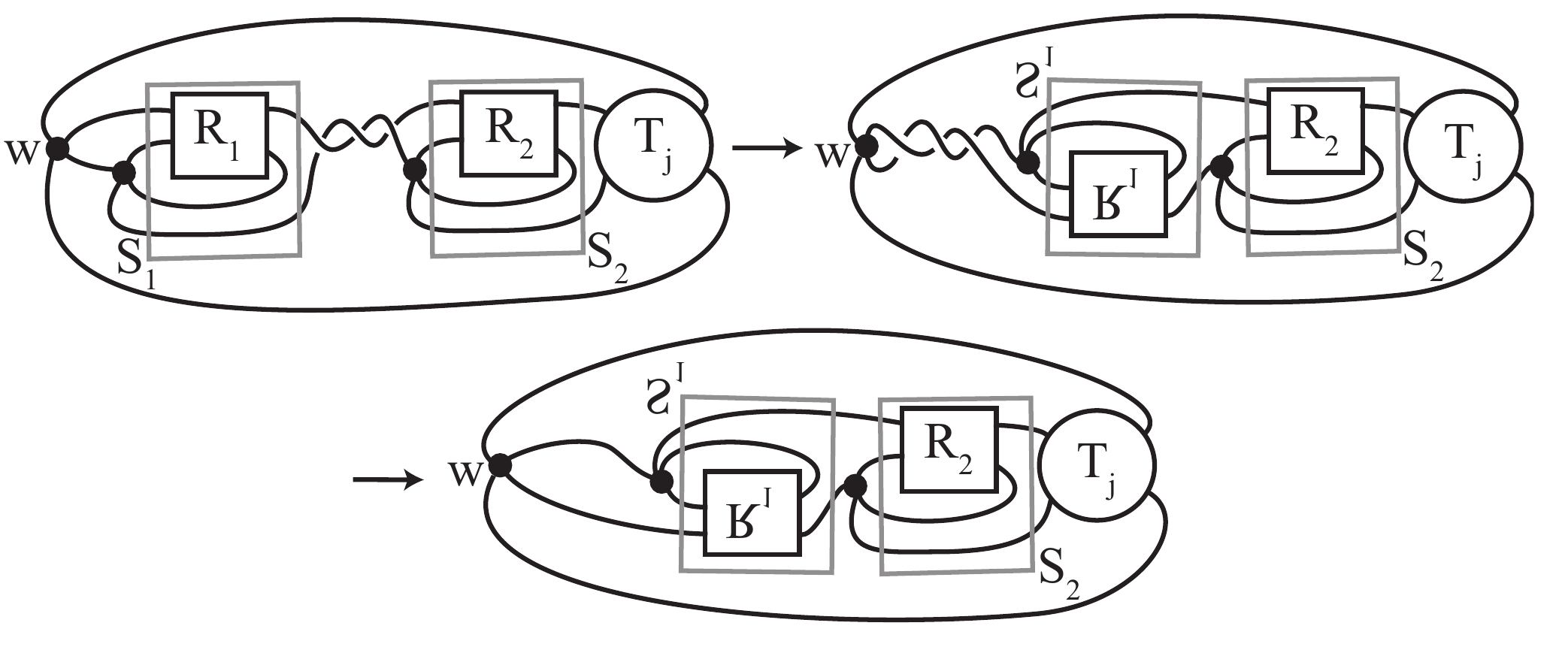}
\caption{We can remove the crossings from the first box of each $T_k'$ with $k\not=j$.}
\label{flypeai}
\end{center}
\end{figure}

After doing the moves illustrated in Figure~\ref{onecrossingina2} and Figure~\ref{flypeai}, there are four different ways that the edges can go in and out of each $R_k$, which we illustrate in Figure~\ref{Possibilities}.

\begin{figure}[http]
\begin{center}
\includegraphics[width=10cm]{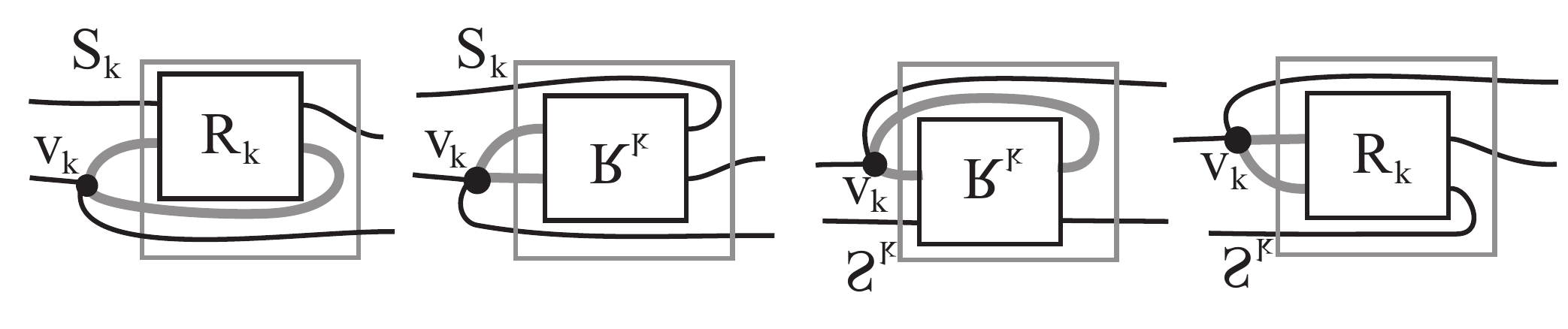}
\caption{The possibilities for how the strands enter and exit each $R_k$, with the loop $c_k$ indicated in grey.}
\label{Possibilities}
\end{center}
\end{figure}

 The leftmost illustration occurs if $v_k$ was originally in the second box so we did not have to flip $R_k$ as in Figure~\ref{onecrossingina2}, and in moving the crossings in the first boxes of the $T_i'$ to the left as in Figure~\ref{flypeai} we flipped $S_k$ zero or an even number of times.  The second illustration in Figure~\ref{Possibilities} occurs if $v_k$ was originally in the third box so we flipped $R_k$ as in Figure~\ref{onecrossingina2}, and in moving the crossings in the first boxes of the $T_i'$ to the left as in Figure~\ref{flypeai} we flipped $S_k$ zero or an even number of times.   The third illustration occurs if $v_k$ was in the second box so we did not flip $R_k$ as in Figure~\ref{onecrossingina2}, but in moving the crossings in the first boxes of the $T_i'$ to the left as in Figure~\ref{flypeai} we flipped $S_k$ an odd number of times.  The rightmost illustration in Figure~\ref{Possibilities} occurs if $v_k$ was in the third box so we flipped $R_k$ as in Figure~\ref{onecrossingina2}, and in moving the crossings in the first boxes of the $T_i'$ to the left as in Figure~\ref{flypeai} we flipped $S_k$ an odd number of times.  

Observe that regardless of which of the four illustrations occur, the only difference between the edges outside of $S_k$ is that the ``dangling edge''  at the left of $S_k$ (that is the one not going into $v_k$) may be above or below the vertex $v_k$.

In Figure~\ref{exceptionalspecific} we define a labeling of the edges of $V(T')$, keeping in mind that $T_j$ has $\infty$-parity and none of the $T_k$ with $k\not=j$ have $\infty$-parity.  In particular, we label the loop containing $v_k$ by $c_k$ and label the edges which are not loops consecutively as follows.  Let $a_1$ be the edge from $w$ to $v_1$, and let $a_2$ be the other edge with one endpoint at $v_1$.  We label the rest of the edges whose vertices are to the left of $T_j$ consecutively from one vertex to the next as $a_3$,\dots, $a_j$.  Then $a_j$ will have one endpoint at $w$, and hence $a=a_1\cup a_2\cup \dots  \cup a_j$ will be a simple closed curve.  Similarly, let $b_n$ be the edge of $V(T')$ from $w$ to the rightmost vertex $v_n$, and then consecutively label the edges whose endpoints are to the right of $T_j$ as $b_{n-1}$, \dots, $b_{j}$.  Then $b_{j}$ will also have an endpoint at $w$.  Thus $b=b_n \cup b_{n-1}\cup \dots \cup b_{j}$ will also be a simple closed curve.

\begin{figure}[http]
\begin{center}      
\includegraphics[width=12cm]{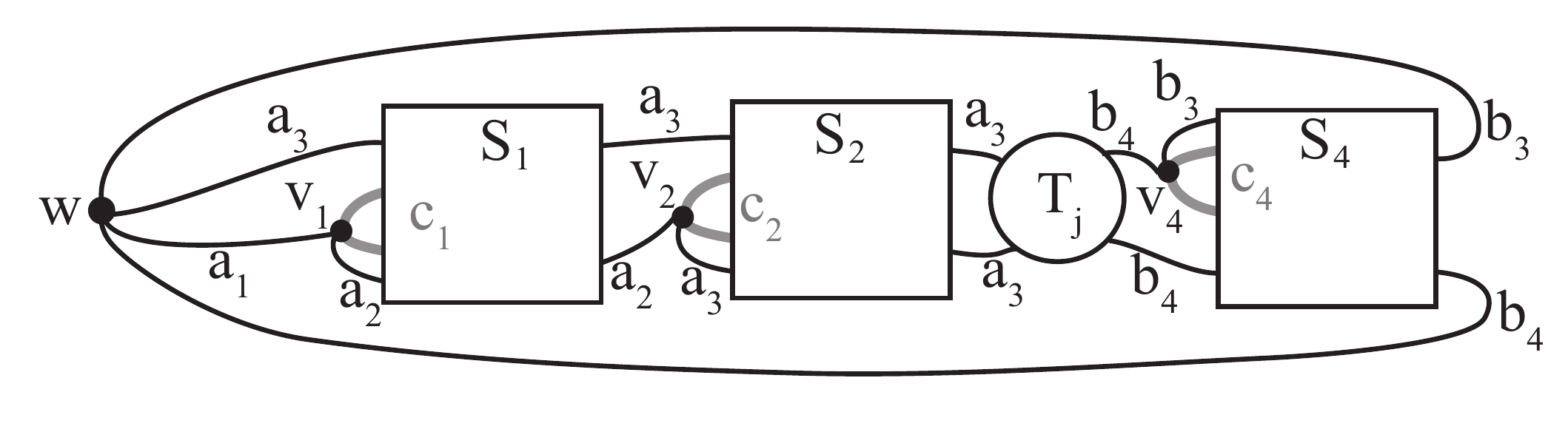}
\caption{We label the edges of $V(T')$ in this way. }
\label{exceptionalspecific}
\end{center}
\end{figure}

 In Figure~\ref{exceptionalspecific}, for $k=1$ and $k=2$ we illustrate the dangling edge at the left of $S_k$  above $v_k$, while for $k=4$ we illustrate the dangling edge at the left of $S_k$ below $v_4$.  In fact, it makes no difference which of these illustrations occur.

Now observe from Figure~\ref{exceptionalspecific} that there are no crossings between the projections of any pair of grey loops $c_k$ and $c_i$ with $i\not =k$, and hence no such pair can be linked.  Also, observe from Figure~\ref{Possibilities} that each individual $c_k$ is the numerator or denominator closure of a single strand of the rational tangle $R_k$, and so must be unknotted. In addition, the loops $a$ and $b$ are each connected sums of numerator or denominator closures of single strands of rational tangles.  Hence $a$ and $b$ are also each unknotted. Finally, $a$ has no crossings with any $c_k$ with $k>j$ and meets every $c_k$ with $k<j$ at the vertex $v_k$.  Hence $a$ cannot be linked with any $c_k$.  Similarly, $b$ cannot be linked with any $c_k$.  It follows that $V(T')$ contains no non-trivial knots or links.

In order to show that $V(T')$ is non-planar we will show that the subgraph $G$ obtained by deleting the loops $c_k$ and vertices $v_k$ for all $k>1$ with $k\not =j$ is non-planar.  The possibilities for $S_k$ with $c_k$ and $v_k$ deleted are illustrated in Figure~\ref{Deleted}.  Observe that since $R_k$ is rational, after the deletion of $c_k$, the tangle $R_k$ is left with a single unknotted strand.  Thus in $G$, each $S_k$ with $k>1$ and $k\not =j$ is a trivial horizontal tangle. 

\begin{figure}[http]
\begin{center}
\includegraphics[width=10cm]{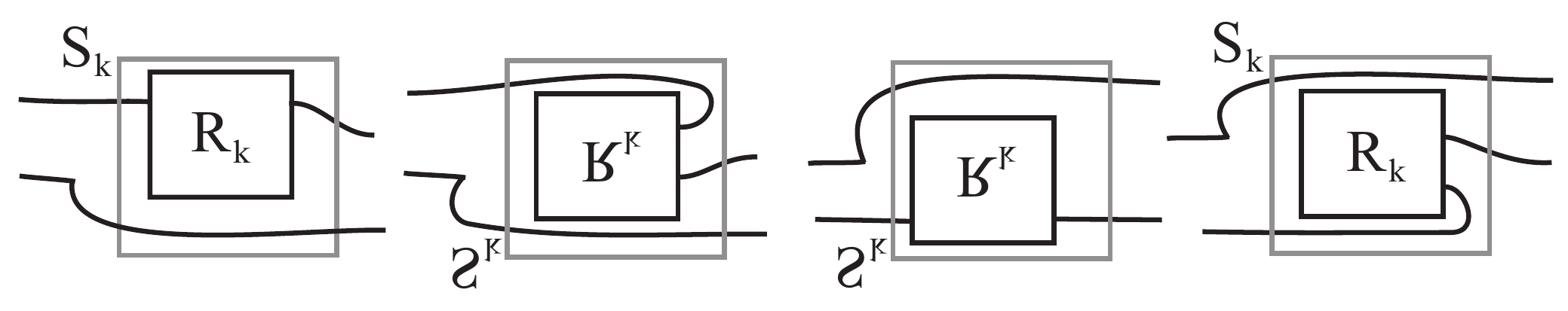}
\caption{The forms of $S_k$ after $c_k$ and $v_k$ have been deleted.}
\label{Deleted}
\end{center}
\end{figure}

\begin{figure}[http]
\begin{center}
\includegraphics[width=10cm]{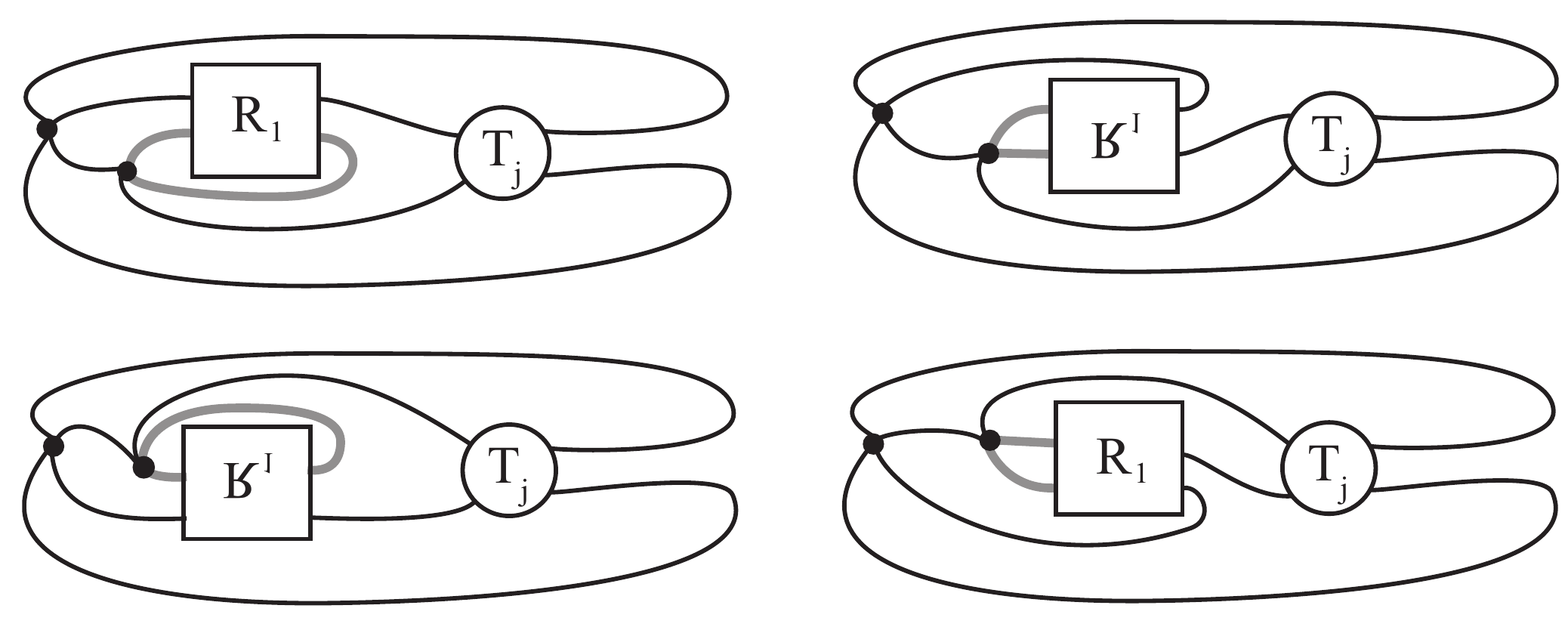}
\caption{$G$ has one of these forms. }
\label{nonplanar}
\end{center}
\end{figure}

 It now follows that the spatial graph $G$ has one of the forms illustrated in Figure~\ref{nonplanar}.  Since $T_j$ has $\infty$-parity, regardless of which form $G$ has, as an abstract graph $G$ is isomorphic to the illustration on the left of Figure~\ref{parity}.   We now let $G_0$ denote the planar embedding of $G$ illustrated on the right side of Figure~\ref{parity}, and obtain the set of associated links $S(G_0)$ by replacing the two vertices of $G_0$ by rational tangles $P$ and $Q$.  

\begin{figure}[http]
\begin{center}
\includegraphics[width=6cm]{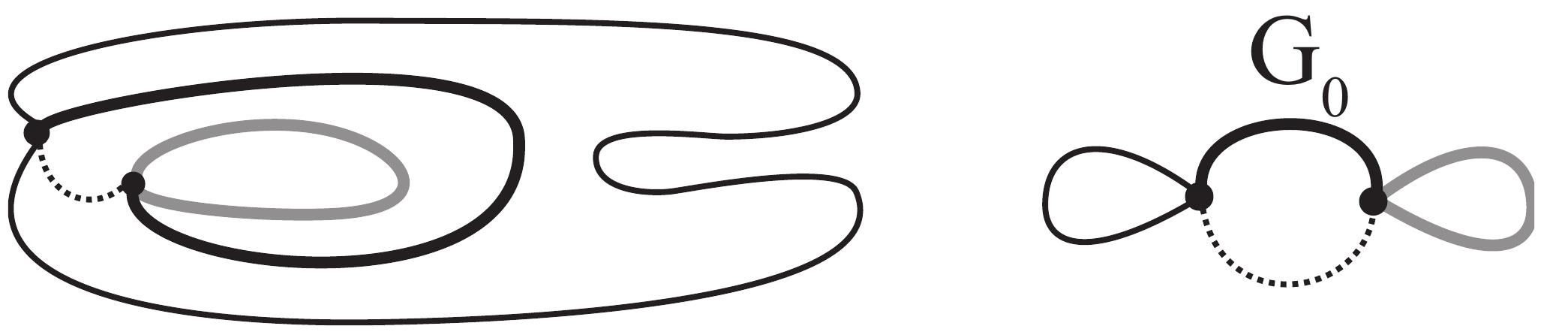}
\caption{$G$ as an abstract graph on the left, and a planar embedding $G_0$ on the right. }
\label{parity}
\end{center}
\end{figure}

The denominator closure of a rational tangle is a $2$-bridge knot or link.  Thus all of the non-trivial, non-split links in $S(G_0)$ are either the connected sum of two 2-bridge knots or links or a single $2$-bridge knot or link. Hence, by Sawollek's Theorem, to show that the spatial graph $G$ is non-planar, it suffices to show that the set of associated links $S(G)$ contains some prime knot or link which is not $2$-bridge.

  In Figure~\ref{PQ}, we replace the vertices $w$ and $v_1$ of $G$ by the rational tangles $P$ and $Q$ to get the elements of $S(G)$.  Then in Figure~\ref{R1Q}, we group the rational tangles $R_1$ and $Q$ together to create a single tangle $U$.   
       
\begin{figure}[http]
\begin{center}
\includegraphics[width=10cm]{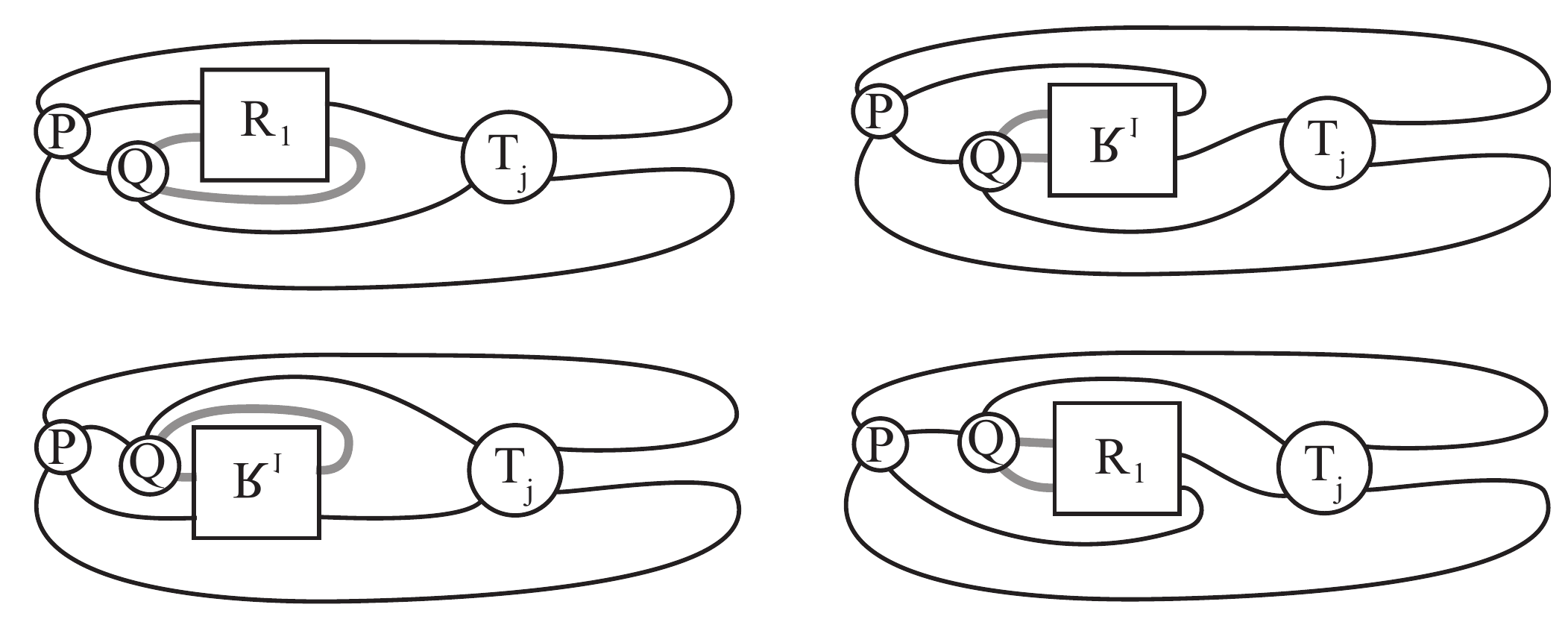}
\caption{The elements of $S(G)$ have one of these forms.}
\label{PQ}
\end{center}
\end{figure}

\begin{figure}[http]
\begin{center}
\includegraphics[width=11cm]{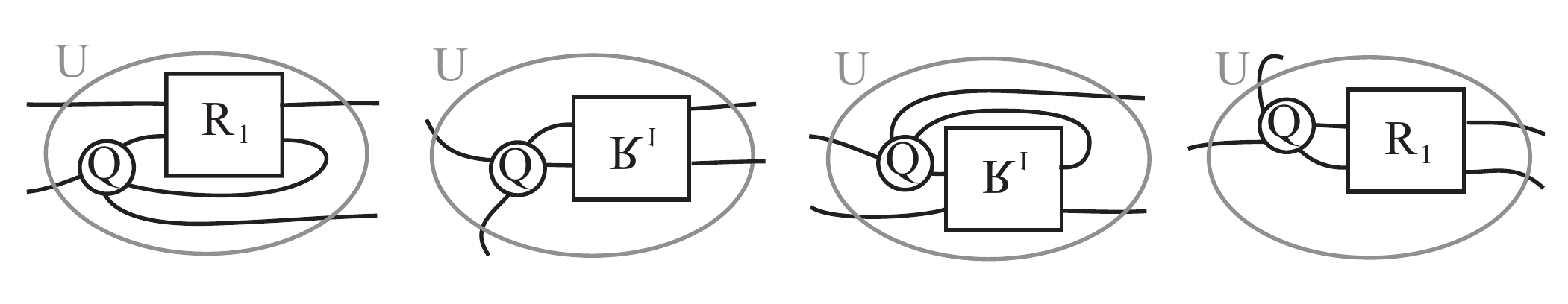}
\caption{We group $R_1$ and $Q$ together into a single tangle $U$. }
\label{R1Q}
\end{center}
\end{figure}

 Recall that $R_1$ is an alternating rational tangle with at least two crossings. Thus we can choose a rational tangle $Q$ so that $U$ is a non-trivial rational tangle which is not horizontal.  Note that the choice of $Q$ will depend on $R_1$ as well as on which form $U$ has.

 Now since $T_j$ has $\infty$-parity, it cannot be horizontal; and by hypothesis $T_j$ cannot be trivial.  Thus for any rational tangle $P$, the knot or link $L=N(P+U+T_j)$ will be the numerator closure of a Montesinos tangle, where neither $U$ nor $T_j$ is horizontal or trivial.

It follows that the double branched cover $\Sigma(L)$ is a Seifert fibered space over $S^2$, and as long as $P$ is not horizontal or trivial, $\Sigma(L)$ has three exceptional fibers.  Now by the classification of Seifert manifolds \cite{OVZ}, we can choose a rational tangle $P$ such that $\Sigma(L)$ is irreducible, not $S^1 \times S^2$, and has infinite fundamental group.  For such a $P$, we know that $L$ will be a prime link which is not $2$-bridge.  Thus $S(G)$ contains a link which is not in $S(G_0)$. It follows that $G$ is non-planar, and hence $V(T')$ must also be non-planar.  Thus we have shown that $V(T')$ is a ravel. \end{proof}

Propositions~\ref{forward} and \ref{backward} together prove Theorem~\ref{monthm}.

\bigskip

\bibliographystyle{hplain}

\end{document}